\documentclass[12pt,leqno,twoside]{amsart}
\usepackage{amssymb,amsmath,amsthm,soul,color}
\usepackage{t1enc}
\usepackage[cp1250]{inputenc}
\usepackage{a4,indentfirst,latexsym}
\usepackage{graphics}
\usepackage{mathrsfs}
\usepackage{cite,enumitem,graphicx}
\usepackage[colorlinks=true,urlcolor=blue,
citecolor=red,linkcolor=blue,linktocpage,pdfpagelabels,
bookmarksnumbered,bookmarksopen]{hyperref}
\usepackage[english]{babel}
\usepackage[left=2.61cm,right=2.61cm,top=2.72cm,bottom=2.72cm]{geometry}
\usepackage[metapost]{mfpic}
\usepackage[colorinlistoftodos]{todonotes}
\usepackage[normalem]{ulem}

\newtheorem{Th}{Theorem}[section]
\newtheorem{Prop}[Th]{Proposition}
\newtheorem{Lem}[Th]{Lemma}
\newtheorem{Cor}[Th]{Corollary}

\newtheorem{Rem}[Th]{Remark}

\newcommand{\wt}{\widetilde}

\newcommand{\vp}{\varphi}

\newcommand{\eps}{\varepsilon}

\def\R{\mathbb{R}}

\def\D{\mathcal{D}}

\def\M{\mathcal{M}}

\def\irn{\int_{\rn}}
\def\rn{\mathbb{R}^N}

\newcommand{\cC}{{\mathcal C}}
\newcommand{\cD}{{\mathcal D}}

\newcommand{\cH}{{\mathcal H}}

\newcommand{\cM}{{\mathcal M}}
\newcommand{\cN}{{\mathcal N}}

\newcommand{\cS}{{\mathcal S}}

\newcommand{\la}{\lambda}
\newcommand{\De}{\Delta}

\newcommand{\Om}{\Omega}

\newcommand{\weakto}{\rightharpoonup}

\numberwithin{equation}{section}

\title[Least energy solutions to a cooperative Schr\"odinger system with $L^2$-bounds]{
	Least energy solutions to a cooperative system of Schr\"odinger equations with prescribed $L^2$-bounds: at least $L^2$-critical growth}

\author[J. Mederski]{Jaros\l aw Mederski}
\address[J. Mederski]{\newline\indent
	Institute of Mathematics,
	\newline\indent 
	Polish Academy of Sciences,
	\newline\indent 
	ul. \'Sniadeckich 8, 00-656 Warsaw, Poland
	\newline\indent 
	and
	\newline\indent 
	Department of Mathematics,
	\newline\indent 
	Karlsruhe Institute of Technology (KIT), 
	\newline\indent 
	D-76128 Karlsruhe, Germany
}
\email{\href{mailto:jmederski@impan.pl}{jmederski@impan.pl}}

\author[J. Schino]{Jacopo Schino}
\address[J. Schino]{\newline\indent
	Institute of Mathematics,
	\newline\indent 
	Polish Academy of Sciences,
	\newline\indent 
	ul. \'Sniadeckich 8, 00-656 Warsaw, Poland
	\newline\indent 
}
\email{\href{mailto:jschino@impan.pl}{jschino@impan.pl}}

\date{\today}

\begin{document}
	\begin{abstract} 
		
		We look for least energy solutions to the cooperative systems of coupled Schr\"odinger equations
		\begin{equation*}
			\begin{cases}
				-\Delta u_i + \lambda_i u_i = \partial_iG(u)\quad \mathrm{in} \ \mathbb{R}^N, \ N \geq 3,\\
				u_i \in H^1(\mathbb{R}^N),\\
				\int_{\mathbb{R}^N} |u_i|^2 \, dx \leq \rho_i^2
			\end{cases}
			i\in\{1,\dots,K\}
		\end{equation*}
		with $G\geq 0$, where $\rho_i>0$ is prescribed and $(\lambda_i, u_i) \in \mathbb{R} \times H^1 (\mathbb{R}^N)$ is to be determined, $i\in\{1,\dots,K\}$.
		Our approach is based on the minimization of the energy functional over a linear combination of the Nehari and Poho\v{z}aev constraints intersected with the product of the closed balls in $L^2(\mathbb{R}^N)$ of radii $\rho_i$, which allows to provide general growth assumptions about $G$ and to know in advance the sign of the corresponding Lagrange multipliers. We assume that $G$ has at least $L^2$-critical growth at $0$ and admits Sobolev critical growth. The more assumptions we make about $G$, $N$, and $K$, the more can be said about the minimizers of the corresponding energy functional. In particular, if $K=2$, $N\in\{3,4\}$, and $G$ satisfies further assumptions, then $u=(u_1,u_2)$ is normalized, i.e., $\int_{\mathbb{R}^N} |u_i|^2 \, dx=\rho_i^2$ for $i\in\{1,2\}$.
		
	\end{abstract}
	
	\maketitle

	\section*{Introduction}
	\setcounter{section}{1}
	
	We consider the following system of autonomous nonlinear Schr\"odinger equations of gradient type
	\begin{equation}
		\label{eq}
		\left\{
		\begin{array}{ll}
			-\Delta u_1 + \lambda_1 u_1 = \partial_1 G(u)\\
			\cdots\\
			-\Delta u_K + \lambda_K u_K = \partial_K G(u)\\
		\end{array} \quad \hbox{in } \rn
		\right.
	\end{equation}
	with $u=(u_1,\dots,u_K)\colon\R^N\to\R^K$, which arises in different areas of mathematical physics. In particular, the system \eqref{eq} describes the propagation of solitons, which are special nontrivial solitary wave solutions $\Phi_j(x,t)=u_j(x)e^{-\mathrm{i}\la_j t}$ to a system of time-dependent Schr\"odinger equations of the form
	\begin{equation}\label{eq:SchrodTime}
		\mathrm{i}\frac{\partial \Phi_j}{\partial t}
		-\Delta \Phi_j = g_j(\Phi)\quad\hbox{for }j=1,\dots,K,
	\end{equation}
	where, for instance,  $g_j$ are responsible for the nonlinear polarization in a photonic crystal \cite{NonlinearPhotonicCrystals,Akozbek} and $\la_j$ are the external electric potentials.
	
	Another field of application is condensed matter
	physics, where \eqref{eq} comes from the system of coupled
	Gross-Pitaevski equations \eqref{eq:SchrodTime} with nonlinearities of the form
	$$g_j(\Phi)=\left(\sum_{k=1}^K\beta_{j,k}|\Phi_k|^2\right)\Phi_j\quad\hbox{for }j=1,\dots,K.$$ 
	The following $L^2$-bounds for $\Phi$ will be studied:
	\[
	\irn\left|\Phi_j(t,x)\right|^2\,dx=\rho_j^2 \quad \text{ and } \quad \irn\left|\Phi_j(t,x)\right|^2\,dx\le\rho_j^2.
	\]
	Problems with prescribed masses $\rho_j^2$ (the former constraint) appear in nonlinear optics, where the mass represents the power supply, and in the theory of Bose--Einstein condensates, where it represents the total number of atoms (see \cite{Akhmediev,Esry,Frantzeskakis,LSSY,Malomed,PiSt,Timmermans}). Prescribing the masses make sense also because they are conserved quantities in the corresponding evolution equation \eqref{eq:SchrodTime} together with the energy (see the functional $J$ below), cf. \cite{CazeLions,Cazenave:book}. As for the latter constraint, we propose it as a model for some experimental 
	situations, e.g. when the power supply provided can oscillate without exceeding a given value.
	
	Recall that a general class of autonomous systems of Schr\"odinger equations was studied by Brezis and Lieb in \cite{BrezisLiebCMP84} and using a constrained minimization method they showed the existence of a {\em least energy solution}, i.e., a nontrivial solution with the minimal energy. Their method using rescaling arguments does not apply with the $L^2$-bounds.
	
	Our aim is to provide a general class of nonlinearities and to find solutions to the nonlinear Schr\"odinger problems
	\begin{equation}\label{eq:}
		\left\{ \begin{array}{l}
			-\Delta u_i + \lambda_i u_i = \partial_iG(u) \quad \mathrm{in} \ \R^N, \ N \geq 3, \\
			u_i \in H^1(\R^N), \\
			\int_{\R^N} |u_i|^2 \, dx \le \rho_i^2
		\end{array} \right. \text{for every }i\in\{1,\dots,K\}
	\end{equation}
	and
	\begin{equation}\label{eq:rho}
		\left\{ \begin{array}{l}
			-\Delta u_i + \lambda_i u_i = \partial_iG(u) \quad \mathrm{in} \ \R^N, \ N \geq 3, \\
			u_i \in H^1(\R^N), \\
			\int_{\R^N} |u_i|^2 \, dx = \rho_i^2
		\end{array} \right. \text{for every }i\in\{1,\dots,K\},
	\end{equation}
	where $\rho=(\rho_1,\dots,\rho_K)\in(0,\infty)^K$ is prescribed and $(\lambda,u) \in\R^K\times H^1(\R^N)^K$ is the unknown.

	Let us introduce the sets
	\begin{equation*}\begin{split}
			\cD & := \left\{ u \in H^1(\R^N)^K \ : \ \int_{\R^N} |u_i|^2 \, dx \leq \rho_i^2 \text{ for every } i\in\{1,\dots,K\} \right\},\\
			\cS & := \left\{ u \in H^1(\R^N)^K \ : \ \int_{\R^N} |u_i|^2 \, dx = \rho_i^2 \text{ for every } i\in\{1,\dots,K\} \right\}
	\end{split}\end{equation*}
	and note that $\cS\subset\partial\cD$.
	
	We shall provide suitable assumptions under which the solutions to \eqref{eq:} (resp. \eqref{eq:rho}) are critical points of the energy functional $J\colon H^1(\R^N)^K \rightarrow \R$ defined as
	$$
	J(u) := \frac12 \int_{\R^N} |\nabla u|^2 \, dx - \int_{\R^N} G(u) \, dx
	$$
	restricted to the constraint $\cD$ (resp. $\cS$) with Lagrange multipliers $\lambda_i \in \R$, i.e., they are critical points of
	$$
	H^1(\R^N)^K \ni u \mapsto J(u) + \frac12 \sum_{i=1}^K \lambda_i \int_{\R^N} |u_i|^2 \, dx \in \R
	$$
	for some $\lambda=(\lambda_1,\dots,\lambda_K) \in \R^K$. Let us recall that, under mild assumptions on $G$, see \cite[Theorem 2.3]{BrezisLiebCMP84}, every critical point of the functional above belongs to $W^{2,q}_\textup{loc}(\R^N)^K$ for all $q < \infty$ and satisfies the Poho\v{z}aev \cite{BerLions1,NonradMed,Jeanjean,Shatah}
	$$
	\int_{\R^N} |\nabla u|^2 \, dx = 2^* \int_{\R^N} G(u) - \frac12 \sum_{i=1}^K \lambda_i |u_i|^2 \, dx
	$$
	and Nehari
	$$
	J'(u)(u)+\sum_{i=1}^K\lambda_i \int_{\R^N} |u_i|^2 \, dx = 0
	$$
	identities. By a linear combination of the two equalities above it is easily checked that every solution satisfies
	$$
	M(u):= \int_{\R^N} |\nabla u|^2 \, dx - \frac{N}{2} \int_{\R^N} H(u) \, dx = 0,
	$$
	where $H(u):= \langle g(u),u\rangle-2G(u)$ ($\langle\cdot,\cdot\rangle$ is the scalar product in $\R^K$) and $g:=\nabla G$, see e.g. \cite{Jeanjean}. Hence we introduce the constraint
	$$
	\cM := \left\{ u \in H^1(\R^N)^K \setminus \{0\}:  M(u)=0 \right\},
	$$
	which contains all the nontrivial solutions to \eqref{eq:} or \eqref{eq:rho} and does not depend on $\la$. Observe that every nontrivial solution to \eqref{eq:} belongs to $\cM\cap\cD$ and every (nontrivial) solution to \eqref{eq:rho} belongs to $\cM \cap \cS\subset\cM \cap \cD$. By a {\em ground state solution} to \eqref{eq:} we mean a nontrivial solution which minimizes $J$ among all the nontrivial solutions. In particular, if $(\lambda,u)$ solves \eqref{eq:} and $J(u)=\inf_{\cM\cap\cD}J$, then $(\lambda,u)$ is a ground state solution (cf. Theorems \ref{th:main1} and \ref{th:main2}). By a {\em ground state solution} to \eqref{eq:rho} we mean that $(\lambda,u)$ solves \eqref{eq:rho} and $J(u)=\inf_{\cM\cap\cD}J$ (cf. Theorems \ref{th:main2}, \ref{th:main3}, and Corollary \ref{Cor}). Note that this is more than just requiring $J(u)=\inf_{\cM\cap\cS}J$, which, on the other hand, appears as a more ``natural'' requirement.
	
	Working with the set $\cD$ instead of the set $\cS$ for a system of Schr\"odinger equations seems to be new and has, among others, a specific advantage related to the sign of the Lagrange multipliers $\la_i$. We begin by showing why this issue is important. First of all, from a physical point of view there are situations, e.g. concerning the eigenvalues of equations describing the behaviour of ideal gases, where the chemical potentials $\la_i$ have to be positive, see e.g. \cite{LSSY,PiSt}. In addition, from a mathematical point of view the (strict) positivity of such Lagrange multipliers often plays an important role in the strong convergence of minimizing sequences in $L^2(\rn)$, see e.g. \cite[Lemma 3.9]{BarJeaSoa}; finally, the nonnegativity is used in some of the proofs below, e.g. the one of Lemma \ref{lem:minons} (a). The aforementioned advantage is as follows: in \cite{Clarke}, Clarke proved that, in a minimization problem, Lagrange multipliers related to a constraint given by inequalities have a sign, i.e., $\lambda_i\ge0$; therefore it is enough to rule out the case $\la_i=0$ in order to prove that $\lambda_i>0$ for every $i\in\{1,\dots,K\}$; note that ruling out the case $\la_i=0$ is simpler than ruling out the case $\la_i\le0$, cf. the proof of Lemma \ref{lem:minons} (b). The nonnegativity/positivity of the Lagrange multipliers of \eqref{eq:rho} has often been obtained by means of involved tools (or at the very minimum in a not-so-straightforward way), such as stronger variants of Palais-Smale sequences in the spirit of \cite{Jeanjean} as in \cite[Lemma 3.6, proof of Theorem 1.1]{BarJeaSoa} or preliminary properties of the ground state energy map $\rho\mapsto\inf_{\cM\cap\cS}J$ as in \cite[Lemma 2.1, proof of Lemma 4.5]{LiZou}. Our argument, based on \cite{Clarke}, is simple, does not seem to be exploited in the theory of normalized solutions, and is demonstrated in Proposition \ref{Prop:Lag} in an abstract way for future applications, e.g. for different operators in the normalized solutions setting like the fractional Laplacian \cite{LuoZhang,LiHeXuYang}.
	
	A second, but not less important, advantage of considering the set $\cD$ concerns the property that the ground state energy in the Sobolev-critical case is below the ground state energy of the limiting problem, cf. \eqref{eq:groundlevel}. More precisely, since in dimension $N\in\{3,4\}$ the Aubin--Talenti instantone is not $L^2$-integrable, we need to truncate it by a cut-off function and then project it into $\cD$; however, unless $K=1$, we cannot ensure that such a projection lies on $\cS$, hence the use of $\cD$ is necessary for this argument. See the proof of Proposition \ref{lem:cless} \textit{(ii)} for further details.
	
	Recall that, when $K=1$ and
	\begin{equation}\label{Ex:powernon}
		G(u)=\frac{1}{p}|u|^{p}, \quad 2<p<2^*,\;p\ne2_N:=2+\frac{4}{N},
	\end{equation}
	\eqref{eq:rho} is equivalent to the corresponding problem with fixed $\lambda > 0$ (and without the $L^2$-bound) via a scaling-type argument. This approach fails in the case of nonhomogeneous nonlinearities or when $K\ge2$.
	In the $L^2$-subcritical case, i.e., when $G(u)\sim|u|^{p}$ with $2 < p < 2_N$, one can obtain the existence of a global minimizer by minimizing directly on $\cS$, cf. \cite{Stuart,Lions84}. In the $L^2$-critical ($p = 2_N$) and the $L^2$-supercritical and Sobolev-subcritical ($2_N < p < 2^*:=\frac{2N}{N-2}$) cases this method does not work; in particular, if $p>2_N$ in \eqref{Ex:powernon}, then $\inf_\cS J=-\infty$. The purpose of this work is to find general growth conditions on $G$ in the spirit of Berestycki, Lions \cite{BerLions1} and Brezis, Lieb \cite{BrezisLiebCMP84} as well as involving the Sobolev critical terms, and to provide a direct approach to obtain ground state solutions to \eqref{eq:}, \eqref{eq:rho}, and similar elliptic problems. 
	The problem \eqref{eq:rho} for one equation was studied by Jeanjean \cite{Jeanjean} and by Bartsch and Soave \cite{BartschSoaveJFA, BartschSoaveJFACorr} with a general nonlinear term satisfying the following condition of Ambrosetti-Rabinowitz type: there exist $\frac4N<a\le b<2^*-2$ such that
	\begin{equation}\label{eq:AR}
		0 < aG(u)\leq H(u)\leq b G(u)\hbox{ for }u \in\R \setminus \{0\}.
	\end{equation}
	In \cite{Jeanjean} the author used a mountain pass argument, while in \cite{BartschSoaveJFA, BartschSoaveJFACorr} a mini-max approach in $\cM$ based on the $\sigma$-homotopy stable family of compact subsets of $\cM$ and the Ghoussoub minimax principle \cite{Ghoussoub} were adopted.
	The same topological principle has been recently applied to the system \eqref{eq:rho} with particular power-like nonlinearities, e.g. in  \cite{BartschSoaveJFA,BartschSoaveJFACorr,BarJeaSoa,BarJea}, and by Jeanjean and Lu \cite{JeanjeanLuNorm} for $K=1$ and a general nonlinearity without \eqref{eq:AR}, but with $L^2$-supercritical growth.

	We stress that the lack of compactness of the embedding $H_\textup{rad}^1(\R^N)\subset L^2(\R^N)$ causes troubles in the analysis of $L^2$-supercritical problems and makes the argument quite involved, see e.g. \cite{Jeanjean,BartschSoaveJFA,BartschSoaveJFACorr}. A possible strategy to recover the compactness of Palais-Smale sequences, at least when $K=1$, is to show that the ground state energy map is nonincreasing with respect to $\rho>0$ and decreasing in a subinterval of $(0,\infty)$, see e.g. \cite{BellazziniJeanjean,JeanjeanLuNorm}. 
	
	In our approach we do not work in $H_\textup{rad}^1$, with Palais-Smale sequences, or with \eqref{eq:AR}, nor the monotonicity of the ground state energy map is required, so that we avoid the mini-max approach in $\cM$ involving a technical topological argument based on \cite{Ghoussoub}, which has been recently intensively exploited by many authors e.g.  in \cite{BartschSoaveJFA,BartschSoaveJFACorr,JeanjeanLuNorm,BarJea,Soave,BarJeaSoa,LiZou,LiHeXuYang,LuoZhang,Soave_crit}.
	
	In particular, we work with a weaker version of \eqref{eq:AR}, see the condition (A5) below, and we admit $L^2$-critical growth at $0$. We make use of a minimizing sequence of $J|_{\cM\cap\cD}$ and we are able to consider a wide class of nonlinearities $G$. In the first part of this work, we adapt the techniques of \cite{BiegMed} to the system \eqref{eq:} and the Sobolev-critical case, which ensure that the minimum of $J$ on $\cM\cap\cD$ is attained. If $G$ is even, we  exploit the Schwarz rearrangement $u^*:=(u_1^*,\dots,u_K^*)$ of $(|u_1|,\dots,|u_K|)$ because, if $u\in\cM\cap\cD$, then $u^*$ can be projected onto the same set without increasing the energy. Next, we point out that dealing with systems \eqref{eq:} and \eqref{eq:rho} one has to involve more tools in order to find a ground state $u\in \cM\cap\partial\cD$ and some additional restrictions imposed on $G$, $N$, or $K$ will be required. In particular, if we want to ensure that the Lagrange multipliers are positive and $u\in\cS$, we use the elliptic regularity results contained in \cite{BrezisLiebCMP84,BerLions1}, the Liouville type result \cite{Ikoma}, and Proposition \ref{Prop:Lag}.
	Finally, a multi-dimensional version of the strict monotonicity of the ground state energy map is simply obtained in Proposition \ref{Prop:GSE} as a consequence of our approach. 
	
	For $2<p\le2^*$, let $C_{N,p}>0$ be the optimal constant in the {\em Gagliardo-Nirenberg inequality}
	\begin{equation}\label{eq:GN}
		|u|_p \leq C_{N,p} |\nabla u|_2^{\delta_p} |u|_2^{1-\delta_p}\quad\hbox{for }u\in H^1(\R^N),
	\end{equation}
	where $\delta_p = N \big( \frac{1}{2} - \frac{1}{p} \big)$ and $\delta_pp>2$ (resp. $\delta_pp=2$, $\delta_pp<2$) if and only if $p>2_N$ (resp. $p=2_N$, $p<2_N$). Here and in what follows we denote by $|u|_k$ the $L^k$-norm of $u$, $1\le k\le\infty$.

We assume there exists $\theta\in(0,\infty)^K$ or $\theta=0$ such that $G$ is of the form
\[
G(u)=\wt G(u)+\frac1{2^*}\sum_{j=1}^K\theta_j|u_j|^{2^*}
\]
for some $\wt G\colon\R^K\to\rn$. We set $\wt g=\nabla G$, $\wt H(u)=\langle\wt g(u),u\rangle-2\wt G(u)$, $\wt h=\nabla \wt H$, $h:=\nabla H$, and consider the following assumptions:  
\begin{itemize}
	\item [(A0)] $\wt g$ and $\wt h$ are continuous and there exists $\tilde{c}>0$ such that $|\wt h(u)|\le\tilde{c}(|u|+|u|^{2^*-1})$.
	\item [(A1)] $\displaystyle\eta:=\limsup_{u\to0}\frac{\wt G(u)}{|u|^{2_N}}<\infty$.
	\item [(A2)] If $\theta=0$, then $\displaystyle\lim_{|u|\to\infty}\frac{\wt G(u)}{|u|^{2_N}}=\infty$; if $\theta\in(0,\infty)^K$, then $\displaystyle\liminf_{|u|\to\infty}\frac{\wt G(u)}{|u|^{2_N}}>0$.
	\item [(A3)] $\displaystyle\lim_{|u|\to\infty}\frac{\wt G(u)}{|u|^{2^*}}=0$.
	\item [(A4)] $\displaystyle2_N \wt H(u)\le\langle \wt h(u),u\rangle$.
	\item [(A5)] $\displaystyle\frac{4}{N} \wt G\le \wt H\le(2^*-2) \wt G$.
\end{itemize}
Of course, $\displaystyle\lim_{|u|\to\infty}\frac{G(u)}{|u|^{2_N}}=\infty$ if (A2) holds and $G,H$ satisfy (A1) -- resp. (A4), (A5) -- if so do $\wt G,\wt H$.	Note that (A5) implies $\wt G, \wt H\ge0$. Note also that $J$ and $M$ are of class $\cC^1$ if (A0) and (A5) are satisfied. For every $u\in H^1(\rn)^K$ such that $\irn H(u)\,dx>0$ we define
	\[
	R:=R_u:=\sqrt{\frac{N\int_{\rn}H(u)\,dx}{2\int_{\rn}|\nabla u|^2\,dx}}>0
	\]
	and note that $u(R\cdot)\in\cM$. 

	Observe that in view of (A2) and (A5), $G(u)\ge\wt G(u)>0$ and $H(u)\ge\wt H(u)>0$ for $u\neq 0$. Indeed, take any $v\in\R^K$ such that $|v|=1$ and note that (A5) implies that
	\begin{equation*}\begin{split}
		 \wt G(v)t^{2^*} \geq \wt G(tv)&\geq \wt G(v)t^{2_N}\quad\hbox{if }t\geq 1,\\
		\wt G(v)t^{2_N} \geq \wt G(tv)&\geq \wt G(v)t^{2^*}\quad\hbox{ if }0<t\leq 1.
	\end{split}\end{equation*}
Since (A2) holds, we get  $\wt G(tv)>0$ for sufficiently large $t>0$, hence taking into account the above inequalities we obtain that $\wt G(tv)>0$ for all $t>0$ and we conclude. In particular, $\cM\neq\emptyset$. Moreover, $\cM$ is a $\cC^1$-manifold, since $M'(u)\neq 0$ for $u\in \cM$, cf. \cite{Shatah}. As a matter of fact, if $M'(u)=0$, then $u$ solves $-\Delta u = \frac{N}{4} h(u)$ and satisfies the Poho\v{z}aev identity $\int_{\R^N} |\nabla u|^2 \, dx = 2^* \frac{N}{4} \int_{\R^N} H(u) \, dx$. If $M(u)=0$, then we infer $u=0$. 
	
	We introduce the following relation:
	\begin{itemize}
		\item[] Let $f_1,f_2\colon\R^K\to\R$. Then $f_1\preceq f_2$ if and only if $f_1\le f_2$ and for every $\eps>0$ there exists $u\in\R^K$, $|u|<\eps$, such that $f_1(u)<f_2(u)$,
	\end{itemize}
	and for better outcomes we need the following stronger variant of (A4): 
\begin{itemize}
	\item [(A4,$\preceq$)] $2_N\wt H(u)\preceq \langle\wt h(u),u\rangle$ if $\theta=0$.
\end{itemize}
Notice that (A4,$\preceq$) implies that	$2_N H(u)\preceq \langle h(u),u\rangle$.

From now on we assume the following condition
\begin{equation}\label{eq:eta2}
2^*C_{N,2_N}^{2_N}\eta|\rho|^{4/N}<1,
\end{equation}
and the first main result concerning \eqref{eq:} reads as follows.
	
	\begin{Th}\label{th:main1}
		Suppose (A0)--(A5) and \eqref{eq:eta2} hold and, if $\theta\in(0,\infty)^K$, 
	\begin{equation}\label{eq:groundlevel}
	\inf_{\M\cap\D}J< \frac{1}{N}S^{N/2}\sum_{i=1}^K\theta_i^{1-N/2}.
	\end{equation}	
		(a) There exists $u\in\cM\cap\cD$ such that
		$J(u)=\inf_{\M\cap\D}J$. In addition, $u$ is a $K$-tuple of radial, nonnegative and radially nonincreasing functions provided that $G$ is of the form
		\begin{equation}\label{eq:Gsp}
			G(u)=\sum_{i=1}^KG_i(u_i)+\sum_{j=1}^L\beta_j\prod_{i=1}^K|u_i|^{r_{i,j}},
		\end{equation}
		where $L\geq 1$, $G_i\colon\R\to[0,\infty)$ is even, $r_{i,j}>1$ or $r_{i,j}=0$, $\beta_j\geq 0$, $2_N\le\sum_{i=1}^Kr_{i,j}<2^*$, and for every $j$ there exists $i_1\neq i_2$ such that $r_{i_1,j}>1$ and $r_{i_2,j}>1$.\\
		(b) If, moreover, (A4,$\preceq$) holds, then $u$ is of class $\cC^2$ and there exists $\la=(\la_1,\dots.\la_K)\in[0,\infty)^K$ such that $(\la,u)$ is a ground state solution to \eqref{eq:}.
	\end{Th}
As we shall see in Section \ref{section:proof}, \eqref{eq:groundlevel} is verified if $N\ge5$ or if $N\in\{3,4\}$ and an additional mild condition holds, see Proposition \ref{lem:cless} (see also Lemma \ref{lem:clessR}). We point out that part \textit{(b)} holds regardless of whether $G$ is of the form \eqref{eq:Gsp} or not. If this is the case, then $u$ has the additional properties as in part \textit{(a)}. 
	
	Notice that (A1) allows $G$ to have $L^{2}$-critical growth $G(u)\sim|u|^{2_N}$ at $0$, but (A2) excludes the same behaviour at infinity. Moreover, $\wt G$ consists of the Sobolev-subcritical part in view of (A3). Finally, the {\em pure $L^2$-critical} case for $|u|$ small is ruled out by (A4,$\preceq$), i.e., $G(u)=\wt G(u)$ cannot be of the form  \eqref{eq:Gsp} with $G_i(u)=\alpha_i |u|^{2_N}$, $\alpha_i \geq 0$, and $\sum_{i=1}^Kr_{i,j}=2_N$ for every $j$. 
	
	Here and later on, when we say $G$ is of the form \eqref{eq:Gsp}, we also mean the additional conditions on $G_i$, $\beta_j$, and $r_{i,j}$ listed in Theorem \ref{th:main1} (a). Observe that $G$ of the form \eqref{eq:Gsp} satisfies (A4) if and only if $G_i$ satisfies the scalar variant of (A4) for all $i\in \{1,\ldots,K\}$. If, in addition, $G_i$ satisfies (A4,$\preceq$) for some $i$, then $G$ satisfies (A4,$\preceq$) as well.

	More can be said if  $N\in\{3,4\}$.

	\begin{Th}\label{th:main2}
		Assume that (A0)--(A3), (A4,$\preceq$), (A5), and \eqref{eq:eta2} are satisfied, $G$ is of the form \eqref{eq:Gsp}, $N\in\{3,4\}$, and \eqref{eq:groundlevel} holds if $\theta\in(0,\infty)^K$. Then there exist $u\in\cM\cap\partial\cD$ of class $\cC^2$ and $\la=(\la_1,\dots,\la_K)\in[0,\infty)^K$ such that $(\la,u)$ is a ground state solution to \eqref{eq:}. In addition, each $u_i$ is radial, nonnegative, and radially nonincreasing.
		Moreover, for every $i\in\{1,\dots,K\}$ either $u_i=0$ or $\int_{\R^N}|u_i|^2\,dx=\rho_i^2$ and, if $u_i\ne0$, then $\la_i>0$ and $u_i>0$. In particular, if $u\in\cS$, then $\lambda\in(0,\infty)^K$ and $(\lambda,u)$ is a ground state solution to \eqref{eq:rho}. 
	\end{Th}
	
	Note that the obtained ground state solution  $u$ belongs to $\partial\cD$, i.e., at least one of the $L^2$-bounds must be the equality $\int_{\R^N} |u_i|^2 \, dx=\rho_i^2$. In particular, ground states solutions can be semitrivial. 
	
	If $K=2$, $L=1$, and the coefficient of the coupling term is large, then we find ground state solutions to \eqref{eq:rho}.

	\begin{Th}\label{th:main3}
		Assume that (A0)--(A3), (A4,$\preceq$), (A5), and \eqref{eq:eta2} are satisfied, $N\in\{3,4\}$, $K=2$, $L=1$, and \eqref{eq:groundlevel} holds if $\theta\in(0,\infty)$. If $G$ is of the form \eqref{eq:Gsp} and $r_{1,1}+r_{2,1}>2_N$, then for every sufficiently large $\beta_1>0$ there exists a ground state solution $(\lambda,u)\in(0,\infty)^2\times\cS$ to \eqref{eq:rho}. Moreover, each component of $u$ is positive, radial, radially nonincreasing and of class $\cC^2$. 
	\end{Th}

	Observe that, if in Theorem \ref{th:main3} $G_i(t)=\mu_i|t|^{p_i}/p_i$ for some $\mu_i>0$ and $p_i\in(2_N,2^*)$, $i\in\{1,2\}$, then clearly $\eta=0$ in \eqref{eq:eta2} and this result was very recently obtained by Li and Zou in \cite[Theorem 1.3]{LiZou}, again, unlike this paper, by means of the involved topological argument due to Ghoussoub \cite{Ghoussoub}, cf. \cite{BartschSoaveJFA,BartschSoaveJFACorr,JeanjeanLuNorm,BarJea,Soave,BarJeaSoa,LiHeXuYang,LuoZhang,Soave_crit}. If $\eta>0$ or $\theta\in(0,\infty)^K$, the result seems to be new and we obtain a ground state solution to \eqref{eq:rho} for sufficiently small $|\rho|$ in the former case, see \eqref{eq:eta2}, or under rather mild additional assumptions about $\wt G$ in the latter, see Proposition \ref{lem:cless}. Furthermore, to our knowledge, this is the first result about normalized solutions to a system of Schr\"odinger equations where the nonlinearity is rather general, in particular not (entirely) of power-type, e.g.
	\begin{equation}\label{Ex1}
	\wt G_i(u)=\frac{\mu_i}{p_i} |u_i|^{p_i}\ln(1+|u_i|),\,\quad p_i\in[2_N,2^*-1],\mu_i>0, i\in\{1,2\}
	\end{equation}
	as well as where the nonlinearity is the sum of power-type nonlinerites including the Sobolev critical terms of the form
	\begin{equation}\label{Ex2}
		G_i(u)=\frac{\nu_i}{2_N}|u_i|^{2_N}+\frac{\mu_i}{p_i}|u_i|^{p_i}+\frac{\theta_i}{2^*}|u_i|^{2^*},\quad p_i\in(2_N,2^*),\mu_i,\nu_i\ge0, \mu_i+\nu_i>0, i\in\{1,2\},
	\end{equation}
	where $\displaystyle\eta=\frac{\max\{\nu_1,\nu_2\}}{2_N}\ge0$.
	In view of Proposition \ref{lem:cless} (ii), taking $p=2_N$ or $p=2^*$ we easily check that \eqref{Ex1} and \eqref{Ex2} satisfy
	\eqref{eq:groundlevel} and we obtain a ground state solution to \eqref{eq:rho} for any $\mu_i+\nu_i>0$ and $\theta_i>0$, $i=1,2$. As for other possible examples of scalar functions $\wt G_1,\wt G_2$ we refer to (E1)--(E4) in \cite{BiegMed}. See also example \eqref{Ex33}.

	
	Moreover, if $K=1$ and $L=0$ (i.e., there is no coupling term), then we find ground state solutions to the scalar problem \eqref{eq:rho} taking into account a general nonlinearity involving at least $L^2$-critical and at most Sobolev-critical growth.
	
	\begin{Cor}\label{Cor}
		Assume that $K=1$, (A0)--(A3), (A4,$\preceq$), (A5), and \eqref{eq:eta2} are satisfied, and \eqref{eq:groundlevel} holds if $\theta\in(0,\infty)$. Assume as well that $H\preceq(2^*-2)G$ or that $N\in\{3,4\}$ and $G$ is even. Then there exist $u\in\cM\cap\cS$ of class $\cC^2$ and $\la\in(0,\infty)$ such that $(\la,u)$ is a ground state solution to \eqref{eq:rho}. If $G$ is even, then $u$ is radial, positive, and radially decreasing.
	\end{Cor}

	Recently, Soave considered \eqref{Ex2} with $\theta_1=0$ in \cite{Soave} and with $\theta_1>0$ but $\nu_1=0$ or  $\mu_1=0$ in \cite{Soave_crit}, with, additionally, an upper bound on $\mu_1>0$ if $N\ge5$.
	In other recent papers, Wei and Wu \cite{WeiWu} considered \eqref{Ex2} with $\theta_1>0$, $\nu_1=0$, and no upper bound on $\mu_1$, while Alves, Ji, and Miyagaki \cite{Alvesetal} considered \eqref{Ex2} with $\theta_1>0$, $\nu_1=0$, and a lower bound on $\mu_1$. Corollary \ref{Cor} generalizes the results from \cite{Alvesetal,Soave_crit,WeiWu} both because no bound on $\mu_1$ is needed (upper or lower) and because the Sobolev-subcritical term $\wt G$ can be $L^2$-critical, $L^2$-supercritical, or even both, without the need of consisting of (sums of) power functions. Of course, Corollary \ref{Cor} also generalizes the results from \cite{BiegMed,JeanjeanLuNorm}, which do not deal with the Sobolev-critical case.

	Finally, observe that conditions (A0)--(A5) and (A4,$\preceq$) are positively additive, i.e., if $\wt G$ and $\wt G'$ satisfy the conditions with $\eta$ and $\eta'$ in (A1) respectively and $\alpha,\alpha'>0$, then $\alpha\wt G+ \alpha'\wt G'$ satisfy the corresponding conditions with $\displaystyle\limsup_{u\to0}\frac{\alpha\wt G+ \alpha'\wt G'}{|u|^{2_N}}\le\alpha\eta+\alpha'\eta'$.

	\section{The proof}\label{section:proof}

	\begin{Lem}\label{Lem:ineq}
		Let $f_1,f_2\in\cC(\R^K)$ and assume there exists $C>0$ such that $|f_1(u)|+|f_2(u)|\le C(|u|^2+|u|^{2^*})$ for every $u\in\R^K$. Then $f_1\preceq f_2$ if and only if $f_1\le f_2$ and
		\[
		\int_{\rn}f_1(u)-f_2(u)\,dx<0
		\]
		for every $u\in H^1(\rn)^K\setminus\{0\}$.
	\end{Lem}
	\begin{proof}
		We argue similarly as in the case $K=1$ provided in \cite[Lemma 2.1]{BiegMed}.
	\end{proof}

	We will always assume that (A0) holds. Lemmas \ref{lem:bdaw1}--\ref{lem:bdaw2} are variants of the results contained in \cite{BiegMed,JeanjeanLuNorm} with some improvements and adapted to the system of equations.
	
	\begin{Lem}\label{lem:bdaw1}
		If (A1)--(A3), (A5),  and \eqref{eq:eta2} hold, then
		$\inf\{|\nabla u|_2^2:u\in\M\cap\D\}>0$.
	\end{Lem}
	\begin{proof}
		Recall that, if $p\in[2,2^*]$, then
		\[
		\big||u|\big|_p=|u|_p \text{ and } \big|\nabla|u|\big|_2\le|\nabla u|_2 \text{ for every }u\in H^1(\rn)^K.
		\]
		For every $\eps>0$ there exists $c_\eps>0$ such that for every $u\in\cM\cap\cD$
		\[\begin{split}
			|\nabla u|_2^2&=\frac{N}{2}\irn H(u)\,dx\le 2^*\bigl(c_\eps|u|_{2^*}^{2^*}+(\eps+\eta)|u|_{2_N}^{2_N}\bigr)=2^*\bigl(c_\eps\big||u|\big|_{2^*}^{2^*}+(\eps+\eta)\big||u|\big|_{2_N}^{2_N}\bigr)\\
			&\le2^*\Bigl(c_\eps C_{N,2^*}^{2^*}\big|\nabla|u|\big|_2^{2^*}+(\eps+\eta)C_{N,2_N}^{2_N}|\rho|^{4/N}\big|\nabla|u|\big|_2^2\Bigr)\\
			&\le2^*\bigl(c_\eps C_{N,2^*}^{2^*}|\nabla u|_2^{2^*}+(\eps+\eta)C_{N,2_N}^{2_N}|\rho|^{4/N}|\nabla u|_2^2\bigr)
		\end{split}\]
		i.e.,
		\begin{equation}\label{eq:Lem22}
			0\le2^*c_\eps C_{N,2^*}^{2^*}|\nabla u|_2^{2^*}+\bigl(2^*(\eps+\eta)C_{N,2_N}^{2_N}|\rho|^{4/N}-1\bigr)|\nabla u|_2^2
		\end{equation}
		Taking $\eps$ sufficiently small so that
		\[
		2^*(\eps+\eta)C_{N,2_N}^{2_N}|\rho|^{4/N}<1
		\]
		we conclude.
	\end{proof}
	
	For $u\in H^1(\rn)^K\setminus\{0\}$ and $s>0$ define $s\star u(x):=s^{N/2}u(sx)$ and $\vp(s):=J(s\star u)$.
	
	\begin{Lem}\label{lem:phi}
		Assume that (A1)--(A5) hold and let $u\in H^1(\rn)^K\setminus\{0\}$ such that
		\begin{equation}\label{eq:eta}
			\eta<\frac{|\nabla u|_2^2}{2|u|_{2_N}^{2_N}}.
		\end{equation}
		Then there exist $a=a(u)>0$ and $b=b(u)\ge a$ such that each $s\in[a,b]$ is a global maximizer for $\vp$ and $\vp$ is increasing on $(0,a)$ and decreasing on $(b,\infty)$. Moreover, $s\star u\in\cM$ if and only if $s\in[a,b]$, $M(s\star u)>0$ if and only if $s\in(0,a)$,
		and $M(s\star u)<0$ is and only if $s>b$. If (A4,$\preceq$) holds, then $a=b$.
	\end{Lem}
	Note that \eqref{eq:eta2} implies \eqref{eq:eta} provided that $u\in\cD$. Indeed, from \eqref{eq:GN}
	$$2\eta|u|_{2_N}^{2_N}\leq 2\eta C_{N,2_N}^{2_N} |\nabla u|_2^{2} |u|_2^{4/N}\leq2\eta C_{N,2_N}^{2_N} |\nabla u|_2^{2}
	|\rho|^{4/N}<|\nabla u|_2^{2}.$$
	\begin{proof}[Proof of Lemma \ref{lem:phi}]
		Notice that from (A1)
		\[
		\vp(s)=\int_{\rn}\frac{s^2}{2}|\nabla u|^2-\frac{G(s^{N/2}u)}{s^N}\,dx\to0
		\]
		as $s\to0^+$ and from (A2) $\lim_{s\to\infty}\vp(s)=-\infty$.
		From (A1) and (A3) for every $\eps>0$ there exists $c_\eps>0$ such that
		\[
		G(u)\le(\eps+\eta)|u|^{2_N}+c_\eps|u|^{2^*},
		\]
		therefore,
		\[
		\vp(s)\geq s^2\Big(\int_{\rn}\frac{1}{2}|\nabla u|^2-(\eta+\eps)|u|^{2_N}\,dx\Big)-c_\eps s^{2^*}\int_{\R^N}|u|^{2^*}\,dx>0
		\]
		for sufficiently small $\eps$ and $s$.	
		It follows that there exists an interval $[a,b]\subset(0,\infty)$ such that $\vp|_{[a,b]}=\max\vp$. Moreover
		\[
		\vp'(s)=s\int_{\rn}|\nabla u|^2-\frac{N}{2}\frac{H(s^{N/2}u)}{s^{N+2}}\,dx
		\]
		and the function
		\[
		s\in(0,\infty)\mapsto\int_{\rn}\frac{H(s^{N/2}u)}{s^{N+2}}\,dx
		\]
		is nondecreasing (resp. increasing) due to (A4) (resp. (A4,$\preceq$) and Lemma \ref{Lem:ineq}) and tends to $\infty$ as $s\to\infty$ due to (A2) and (A5). There follows that $\vp'(s)>0$ if $s\in(0,a)$ and $\vp'(s)<0$ if $s>b$ and that $a=b$ if (A4,$\preceq$) holds. Finally, observe that
		\[
		s\vp'(s)=\int_{\rn}s^2|\nabla u|^2-\frac{N}{2}\frac{H(s^{N/2}u)}{s^N}\,dx=M(s\star u).\qedhere
		\]
	\end{proof}
	
	\begin{Lem}\label{lem:coerc}
		If (A1)--(A5) and \eqref{eq:eta2} are verified, then $J$ is coercive on $\cM\cap\cD$.
	\end{Lem}
	\begin{proof}
		First of all note that, if $u\in\cM$, then due to (A5)
		\[
		J(u)=J(u)-\frac12M(u)=\int_{\rn}\frac{N}{4}H(u)-G(u)\,dx\ge0
		\]
		and so, a fortiori, $J$ is nonnegative on $\cM\cap\cD$.
		Let $(u^{(n)})\subset\cM\cap\cD$ such that $\|u^{(n)}\|\to\infty$, i.e., $\lim_n|\nabla u^{(n)}|_2=\infty$, and define
		\[
		s_n:=|\nabla u^{(n)}|_2^{-1}>0\quad\text{and}\quad w^{(n)}:=s_n\star u^{(n)}.
		\]
		Note that $s_n\to0$, $|w^{(n)}_i|_2=|u^{(n)}_i|_2\le\rho_i$ for $i\in\{1,\dots,K\}$, and $|\nabla w^{(n)}|_2^2=1$, in particular $(w^{(n)})$ is bounded in $H^1(\rn)^K$.
		Suppose by contradiction that
		\[
		\limsup_n\max_{y\in\rn}\int_{B(y,1)}|w^{(n)}|^2\,dx>0.
		\]
		Then there exist $(y^{(n)})\subset\rn$ and $w\in H^1(\rn)^K$ such that, up to a subsequence, $w^{(n)}(\cdot+y^{(n)})\rightharpoonup w\ne0$ in $H^1(\rn)^K$ and $w^{(n)}(\cdot+y^{(n)})\to w$ a.e. in $\rn$. Thus, owing to (A2),
		\[\begin{split}
			0&\le\frac{J(u^{(n)})}{|\nabla u^{(n)}|_2^2}\le\frac12-\int_{\rn}\frac{G(u^{(n)})}{|\nabla u^{(n)}|_2^2}\,dx=\frac12-s_n^{N+2}\int_{\rn}G\bigl(u^{(n)}(s_nx)\bigr)\,dx\\
			&=\frac12-s_n^{N+2}\int_{\rn}G(s_n^{-N/2}w^{(n)})=\frac12-\int_{\rn}\frac{G(s_n^{-N/2}w^{(n)})}{|s_n^{-N/2}w^{(n)}|^{2_N}}|w^{(n)}|^{2_N}\,dx\\
			&=\frac12-\int_{\rn}\frac{G\bigl(s_n^{-N/2}w^{(n)}(x+y^{(n)})\bigr)}{|s_n^{-N/2}w^{(n)}(x+y^{(n)})|^{2_N}}|w^{(n)}(x+y^{(n)})|^{2_N}\,dx\to-\infty.
		\end{split}\]
		It follows that
		\[
		\lim_n\max_{y\in\rn}\int_{B(y,1)}|w^{(n)}|^2\,dx=0
		\]
		and so, from Lions' Lemma \cite{Lions84}, $w^{(n)}\to0$ in $L^{2_N}(\rn)^K$. Since $$s_n^{-1}\star w^{(n)}=u^{(n)}\in\cM,$$ Lemma \ref{lem:phi} yields
		\[
		J(u^{(n)})=J(s_n^{-1}\star w^{(n)})\ge J(s\star w^{(n)})=\frac{s^2}{2}-s^N\int_{\rn}G\bigl(s^{N/2}w^{(n)}(s\cdot)\bigr)\,dx
		\]
		for every $s>0$. Taking into account that
		\[
		\lim_n\int_{\rn}G\bigl(s^{N/2}w^{(n)}(s\cdot)\bigr)\,dx=0,
		\]
		we have that $\liminf_nJ(u^{(n)})\ge s^2/2$ for every $s>0$, i.e., $\lim_nJ(u^{(n)})=\infty$.
	\end{proof}
	
	\begin{Lem}\label{lem:bdaw2}
		If (A1)--(A5) and \eqref{eq:eta2} are verified, then $c:=\inf_{\cM\cap\cD}J>0$.
	\end{Lem}
	\begin{proof}
		We prove that there exists $\alpha>0$ such that
		\begin{equation}\label{eq:alpha}
			|\nabla u|_2\le\alpha\Rightarrow J(u)\ge\frac{|\nabla u|_2^2}{2N}.
		\end{equation}
		From \eqref{eq:GN} and \eqref{eq:eta2}, for every $\eps>0$ there exists $c_\eps>0$ such that
		\[\begin{split}
			\int_{\rn}G(u)\,dx&\le c_\eps C_{N,2^*}^{2^*}|\nabla u|_2^{2^*}+(\eps+\eta)C_{N,2_N}^{2_N}|\rho|^{4/N}|\nabla u|_2^2\\
			&\le\left(c_\eps C_{N,2^*}^{2^*}|\nabla u|_2^{2^*-2}+\eps C_{N,2_N}^{2_N}|\rho|^{4/N}+\frac12-\frac1N\right)|\nabla u|_2^2.
		\end{split}\]
		Choosing
		\[
		\eps=\frac{1}{4NC_{N,2_N}^{2_N}|\rho|^{4/N}} \quad \text{and} \quad \alpha=\frac{1}{(4Nc_\eps C_{N,2^*}^{2^*})^\frac{1}{2^*-2}}
		\]
		we obtain, provided $|\nabla u|_2\le\alpha$,
		\[
		\int_{\rn}G(u)\,dx\le\left(\frac12-\frac{1}{2N}\right)|\nabla u|_2^2
		\]
		and so $\displaystyle J(u)\ge\frac{|\nabla u|_2^2}{2N}$.
		Now take $u\in\cM\cap\cD$ and $\alpha>0$ such that \eqref{eq:alpha} holds and define
		\[
		s:=\frac{\alpha}{|\nabla u|_2}\quad\text{and}\quad w:=s\star u.
		\]
		Clearly $|w_i|_2=|u_i|_2\le\rho_i$ for $i\in\{1,\dots,K\}$ and $|\nabla w|_2=\alpha$, whence in view of Lemma \ref{lem:phi}
		\[
		J(u)\ge J(w)\ge\frac{|\nabla w|_2^2}{2N}=\frac{\alpha^2}{2N}>0.\qedhere
		\]
	\end{proof}

From now on, $c>0$ will stand for the infimum of $J$ over $\cM\cap\cD$.

\begin{Prop}\label{lem:cless}
	Assume that $\theta\in(0,\infty)^K$ and that (A1)--(A5), \eqref{eq:eta2}, and (at least) one of the following conditions hold:
	\begin{itemize}
		\item [(i)] $N\geq 5$;
		\item [(ii)] there exist $2_N\leq p\leq 2^*$ and $2_N\le q<2^*$ such that
		\begin{equation}\label{eq:pq}
		\liminf_{|u|\to 0}\frac{\wt G(u)}{|u|^p}>0 \quad \text{and} \quad \liminf_{|u|\to \infty}\frac{\wt G(u)}{|u|^q}>0
		\end{equation}
		and $\max\{p,q\}/2-\min\{p,q\}<-1$ if $N=3$.
	\end{itemize}
	Then \eqref{eq:groundlevel} holds.
\end{Prop}

Recall that, from (A2), the second condition in \eqref{eq:pq} always holds with $q=2_N$. Notice that the restriction on the relation between $p,q$ is always satisfied if $p=q$.

	\begin{proof}[Proof of Proposition \ref{lem:cless}]
		Define $u_0^1$ as the Aubin--Talenti instanton \cite{Aubin,Talenti}
		\[
		u_0^1(x):=\left(\frac{\sqrt{N(N-2)}}{1+|x|^2}\right)^\frac{N-2}{2}
		\]
		and, for $\varepsilon>0$,
		\[
		u_0^\varepsilon(x):=\varepsilon^{1-N/2}u_0^1(x/\varepsilon)=\left(\frac{\varepsilon\sqrt{N(N-2)}}{\varepsilon^2+|x|^2}\right)^\frac{N-2}{2}.
		\]
		Recall that, for every $\varepsilon>0$, $|\nabla u_0^\varepsilon|_2=|\nabla u_0^1|_2$, $|u_0^\varepsilon|_{2^*}=|u_0^1|_{2^*}$, and $u_0^\varepsilon$ is a minimizer for
		\begin{equation*}
		S:=\inf\left\{\int_{\rn}|\nabla v|^2\,dx:v\in \cD^{1,2}(\rn),\int_{\rn}|v|^{2^*}\,dx=1\right\}.
		\end{equation*}
	
		\textit{(i)} For every $\eps>0$ and $j\in\{1,\dots,K\}$ define $\bar{u}_j^\varepsilon:=\theta_j^{(2-N)/4}u_0^\eps$. Since $u_0^\varepsilon\in L^2(\rn)$ for every $\varepsilon>0$ and $|u_0^\varepsilon|_2\to0$ as $\varepsilon\to0^+$, we have $\bar{u}^\varepsilon:=(\bar{u}^\varepsilon_1,...,\bar{u}^\varepsilon_K)\in\cD$ for sufficiently small $\varepsilon$. Moreover, in view of Lemma \ref{lem:BarS}, $\bar{u}^\eps$ is such that
		\[
		\frac{|\nabla \bar{u}^\varepsilon|_2^2}{\left(\sum_{j=1}^K\theta_j|\bar{u}_j^\varepsilon|_{2^*}^{2^*}\right)^{2/2^*}}=\inf_{u\in\cD^{1,2}(\R^N)^K\setminus\{0\}}\frac{|\nabla u|_2^2}{\left(\sum_{j=1}^K\theta_j|u_j|_{2^*}^{2^*}\right)^{2/2^*}}=\left(\sum_{j=1}^K\theta_j^{1-N/2}\right)^{2/N}S.
		\]
Recall that $\wt G(u)>0$ for $u\neq 0$ and then, taking $\varepsilon$ sufficiently small,
		\begin{equation*}\begin{split}
		c\leq J(s_\varepsilon\star \bar{u}^\eps) & \leq-\int_{\R^N} \wt G(s_\varepsilon\star \bar{u}^\eps) \,dx+ \max_{s>0}\frac{s^2}{2}\int_{\rn}|\nabla \bar{u}^\eps|^2\,dx-\frac{s^{2^*}}{2^*} \sum_{j=1}^K\theta_j\int_{\rn}|\bar{u}_j^\eps|^{2^*}\,dx\\
		& <\max_{s>0}\frac{s^2}{2}\int_{\rn}|\nabla \bar{u}^\eps|^2\,dx-\frac{s^{2^*}}{2^*} \sum_{j=1}^K\theta_j\int_{\rn}|\bar{u}_j^\eps|^{2^*}\,dx\\
		& =\frac{1}{N}\frac{|\nabla \bar{u}^\varepsilon|_2^N}{\left(\sum_{j=1}^K\theta_j|\bar{u}_j^\varepsilon|_{2^*}^{2^*}\right)^{N/2-1}}=\sum_{j=1}^K\theta_j^{1-N/2}\frac{S^{N/2}}{N}.
		\end{split}\end{equation*}
	
\textit{(ii)} If $N\ge5$, then the statement follows form \textit{(i)}, therefore we can assume $N\in\{3,4\}$. Since $u_0^1\not\in L^2(\rn)$, let $0\le\phi\in\cC_0^\infty(\rn)$ radial such that $\phi\equiv1$ in $B_1$ and $\phi\equiv0$ in $\rn\setminus B_2$, where $B_r$ stands for the closed ball centred at $0$ of radius $r$. For every $\varepsilon>0$ define
\[
u_j^\varepsilon:=\theta_j^\frac{2-N}{4}\phi u_0^\varepsilon \quad \text{and} \quad v^\varepsilon:=\frac{\bar\rho}{|u^\varepsilon|_2}(u_1^\varepsilon,\dots,u_K^\varepsilon)\in\cD,
\]
where $\bar{\rho}:=\min_{j\in\{1,\dots,K\}}\rho_j$, and recall (cf., e.g., \cite[p. 179]{Struwe}, \cite[Lemma A.1]{Soave_crit}) that
\begin{equation*}\begin{split}
|\nabla(\phi u_0^\varepsilon)|_2^2 & =S^{N/2}+O(\varepsilon^{N-2})\\
|\phi u_0^\varepsilon|_{2^*}^2 & =
\begin{cases}
S+O(\varepsilon^4) \quad \text{if } N=4\\
S^{1/2}+O(\varepsilon^2) \quad \text{if } N=3
\end{cases}\\
|\phi u_0^\varepsilon|_2^2 & =
\begin{cases}
C_4\varepsilon^2|\ln\varepsilon|+O(\varepsilon^2) \quad \text{if } N=4\\
C_3\varepsilon+O(\varepsilon^2) \quad \text{if } N=3,
\end{cases}
\end{split}\end{equation*}
where $C_N>0$ depends only on $N$ and $\phi$.
Note that
\begin{equation*}
\int_{\R^N}|\phi u_0^\varepsilon|^r\chi_{\{\phi u_0^\varepsilon\geq 1\}}\,dx\geq C\varepsilon^{N-(N/2-1)r}
\end{equation*}
for some constant $C>0$ and sufficiently small $\eps>0$, where $r\in\{p,q\}$ and $\chi_A$ stands for the characteristic function of $A$.
Indeed, let $|x|^2\le\eps\sqrt{N(N-2)}-\eps^2$. If $\eps$ is sufficiently small, then $x\in B_1$ and, consequently, $\phi(x) u_0^\varepsilon(x)=u_0^\eps(x)\geq 1$, whence
\begin{equation*}\begin{split}
\int_{\R^N}|\phi u_0^\varepsilon|^r\chi_{\{\phi u_0^\varepsilon\geq 1\}}\,dx&\geq
\int_{\left\{|x|\le\left(\eps\sqrt{N(N-2)}-\eps^2\right)^{1/2}\right\}}|u^\varepsilon_0|^r\,dx\\
&=\eps^{N-(N/2-1)r}\int_{\left\{|y|\le\left(\sqrt{N(N-2)}/\eps-1\right)^{1/2}\right\}}|u^1_0|^r\,dy
\end{split}\end{equation*}
and we conclude, since $u^1_0\in L^r(\R^N)$. Define $s_\varepsilon>0$ such that $s_\varepsilon\ast v^\varepsilon\in\cM$. In a similar way to the proof of Lemma \ref{lem:bdaw1}, for every $\delta>0$ there exists $C_\delta>0$ not depending on $\varepsilon$ such that
\[
\frac{1}{2^*}|\nabla v^\varepsilon|_2^2\le(\eta+\delta)|v^\varepsilon|_{2_N}^{2_N}+C_\delta s_\varepsilon^{2^*-2}\sum_{j=1}^K\theta_j|v_j^\varepsilon|_{2^*}^{2^*}\le(\eta+\delta)C_{N,2_N}^{2_N}|\rho|^{4/N}|\nabla v^\varepsilon|_2^2+C_\delta s_\varepsilon^{2^*-2}\sum_{j=1}^K\theta_j|v_j^\varepsilon|_{2^*}^{2^*}
\]
(note that $u\mapsto\left(\sum_{j=1}^K\theta_j|u_j|_{2^*}^{2^*}\right)^{1/2^*}$ is an equivalent norm in $L^{2^*}(\rn)^K$, i.e., taking $\delta$ sufficiently small and denoting $m:=\bigl(1/2^*-(\eta+\delta)C_{N,2_N}^{2_N}|\rho|^{4/N}\bigr)/C_\delta>0$,
\[
s_\varepsilon^{2^*-2}\ge \frac{m|\nabla v^\varepsilon|_2^2}{\sum_{j=1}^K\theta_j|v_j^\varepsilon|_{2^*}^{2^*}}=m\bar{\rho}^{2-2^*}\frac{|\nabla(\phi u_0^\eps)|_2^2|u^\eps|_2^{2^*-2}}{|\phi u_0^\eps|_{2^*}^{2^*}}.
\]
In a similar way to point \textit{(i)},
\[
c\le-\irn\wt G(s_\varepsilon\star v^\eps)\,dx+\frac1N\frac{|\nabla v^\eps|_2^N}{\left(\sum_{j=1}^K\theta_j|v_j^\eps|_{2^*}^{2^*}\right)^{N/2-1}}.
\]
There holds
\[\begin{split}
|\nabla v^\eps|_2^2 = \frac{\bar{\rho}^2|\nabla(\phi u_0^\eps)|_2^2\sum_{j=1}^K\theta_j^{1-N/2}}{|u^\eps|_2^2} \quad \text{and} \quad \sum_{j=1}^K\theta_j|v_j^\eps|_{2^*}^{2^*}= \frac{\bar{\rho}^{2^*}|\phi u_0^\eps|_{2^*}^{2^*}\sum_{j=1}^K\theta_j^{1-N/2}}{|u^\eps|_2^{2^*}},
\end{split}\]
thus, denoting $k=2$ (resp. $k=4$) if $N=3$ (resp. $N=4$),
\[\begin{split}
\frac{|\nabla v^\eps|_2^N}{\left(\sum_{j=1}^K\theta_j|v_j^\eps|_{2^*}^{2^*}\right)^{N/2-1}} & =\sum_{j=1}^K\theta_j^{1-N/2}\left(\frac{|\nabla(\phi u_0^\eps)|_2}{|\phi u_0^\eps|_{2^*}}\right)^N=\sum_{j=1}^K\theta_j^{1-N/2}\left(\frac{S^{N/2}+O(\varepsilon^{N-2})}{S^{(N-2)/2}+O(\varepsilon^k)}\right)^{N/2}\\
& =\sum_{j=1}^K\theta_j^{1-N/2}\bigl(S+O(\eps^{N-2})\bigr)^{N/2}=\sum_{j=1}^K\theta_j^{1-N/2}S^{N/2}+O(\eps^{N-2}).
\end{split}\]
Now we estimate $\int_{\rn}\wt G(s_\varepsilon\star v^\varepsilon)\,dx$ as $\varepsilon\to0^+$. From \eqref{eq:pq} and the fact, due to (A2) and (A5), that $\wt G(u)>0$ if $u\ne0$, we deduce there exists $C>0$ such that $\wt G(u)\ge C|u|^p$ if $|u|\le1$ and $\wt G(u)\ge C|u|^q$ if $|u|>1$.
\begin{equation*}\begin{split}
\int_{\rn}\wt G(s_\varepsilon\star v^\varepsilon)\,dx&\geq C s_\varepsilon^{N(p/2-1)}\int_{\R^N}|v^\varepsilon|^p\chi_{\{|s_\eps^{N/2} v^\varepsilon|\leq 1\}}\,dx
+C s_\varepsilon^{N(q/2-1)}\int_{\R^N}|v^\varepsilon|^q\chi_{\{|s_\eps^{N/2} v^\varepsilon|> 1\}}\,dx\\
&\geq C' |\phi u_0^\varepsilon|_2^{N(p/2-1)-p}\int_{\R^N}|\phi u_0^\varepsilon|^p\chi_{\{|s_\eps^{N/2} v^\varepsilon|\leq 1\}}\,dx\\
&\hspace{5mm}+C' |\phi u_0^\varepsilon|_2^{N(q/2-1)-q}\int_{\R^N}|\phi u_0^\varepsilon|^q\chi_{\{|s_\eps^{N/2} v^\varepsilon|> 1\}}\,dx\\
&\geq C' |\phi u_0^\varepsilon|_2^{(N/2-1)p-N}\int_{\R^N}|\phi u_0^\varepsilon|^p\chi_{\{|s_\eps^{N/2}v^\varepsilon|\leq 1\}}\chi_{\{\phi u_0^\varepsilon\geq 1\}}\,dx\\
&\hspace{5mm}+C' |\phi u_0^\varepsilon|_2^{(N/2-1)q-N}\int_{\R^N}|\phi u_0^\varepsilon|^q\chi_{\{|s_\eps^{N/2} v^\varepsilon|> 1\}}\chi_{\{\phi u_0^\varepsilon\geq 1\}}\,dx\\
&\geq C' \min\big\{|\phi u_0^\varepsilon|_2^{(N/2-1)p-N},|\phi u_0^\varepsilon|_2^{(N/2-1)q-N}\big\}\int_{\R^N}|\phi u_0^\varepsilon|^{\min\{p,q\}}\chi_{\{\phi u_0^\varepsilon\geq 1\}}\,dx\\
&\geq C'' |\phi u_0^\varepsilon|_2^{(N/2-1)\max\{p,q\}-N}\eps^{N-(N/2-1)\min\{p,q\}}
\end{split}\end{equation*}
as $\eps\to0^+$ because $(N/2-1)r-N<0$, $r\in\{p,q\}$,  where $C',C''>0$ are constants. There follows that
\[
c\le\sum_{j=1}^K\theta_j^{1-N/2}\frac{S^{N/2}}{N}+O(\eps^{N-2})-C''|\phi u_0^\varepsilon|_2^{(N/2-1)\max\{p,q\}-N}\eps^{N-(N/2-1)\min\{p,q\}}.
\]
If $N=3$, then
\[\begin{split}
|\phi u_0^\varepsilon|_2^{(N/2-1)\max\{p,q\}-N}\eps^{N-(N/2-1)\min\{p,q\}} & =\frac{\eps^{3-\min\{p,q\}/2}}{(C_3\eps)^{3/2-\max\{p,q\}/4}+O(\eps^{3-\max\{p,q\}/2})}\\
& \ge C\eps^{(3+\max\{p,q\}/2-\min\{p,q\})/2}
\end{split}\]
and $0<(3+\max\{p,q\}/2-\min\{p,q\})/2<1=N-2$. If $N=4$, then
\[\begin{split}
|\phi u_0^\varepsilon|_2^{(N/2-1)\max\{p,q\}-N}\eps^{N-(N/2-1)\min\{p,q\}} & =\frac{\eps^{4-\min\{p,q\}}}{(\sqrt{C_4|\ln\eps|}\,\eps)^{4-\max\{p,q\}}+O(\eps^{4-\max\{p,q\}})}\\
& \ge C\eps^{|p-q|}|\ln\eps|^{\max\{p,q\}/2-2}
\end{split}\]
and $|p-q|<2=N-2$, $\max\{p,q\}-4\le0$, and $|p-q|>0$ or $\max\{p,q\}-4<0$. Either way, $O(\varepsilon^{N-2})-C''|\phi u_0^\varepsilon|_2^{(N/2-1)\max\{p,q\}-N}\eps^{N-(N/2-1)\min\{p,q\}}<0$  for sufficiently small $\varepsilon$ and
\[
c<\sum_{j=1}^K\theta_j^{1-N/2}\frac{S^{N/2}}{N}.\qedhere
\]
\end{proof}

Since there exist nonlinearities that do not satisfy the assumptions of Proposition \ref{lem:cless} \textit{(ii)}, we provide other sufficient conditions for \eqref{eq:groundlevel} to hold.

	\begin{Lem}\label{lem:clessR}
	Assume that (A1)--(A5) are satisfied and $\theta\in(0,\infty)^K$.\\
	(a) If $K=1$, $\eta=0$, and $\displaystyle\lim_{u\to0}\wt G(u)/|u|^{2^*}=\infty$, then there exists $\rho_0>0$ such that \eqref{eq:groundlevel} is satisfied provided that $\rho>\rho_0$.\\
	(b) If \eqref{eq:eta2} holds and  $\displaystyle\lim_{|u|\to\infty}\wt G(u)/|u|^{2_N}=\infty$, then there exists $\theta_0>0$ such that \eqref{eq:groundlevel} is satisfied provided that $\theta_i<\theta_0$ for some $i\in\{1,...,K\}$.
	\end{Lem}
	\begin{proof}
	\textit{(a)} We 
	prove that $c\to0$ as $\rho\to\infty$ (note that \eqref{eq:eta2} is satisfied for every $\rho>0$ because $\eta=0$). Let $\rho_n \to \infty$ and take $u \in L^\infty(\R^N)$ such that $|u|_2=1$. Without loss of generality we may assume that $\rho_n > 1$ and
	define $u_n := \rho_n u$ so that $|u_n|_2=\rho_n$. From Lemma \ref{lem:phi} there exists $s_n >0$ such that $v_n :=s_n^{N/2} u_n (s_n \cdot) \in \cM$. Moreover, $|v_n|_2=|u_n|_2$, hence
	$$
	0 < \inf\{J(v) : v\in\cM, \, |v|_2\le\rho_n\} \leq J(v_n) \leq \frac12 \int_{\R^N} |\nabla v_n|^2 \, dx = \frac12 (s_n\rho_n)^2  \int_{\R^N} |\nabla u|^2 \, dx,
	$$
	so it is enough to show that $s_n \rho_n \to 0$. Note that
$$
\left(s_n\rho_n\right)^2 \int_{\R^N} |\nabla u|^2 \, dx = \int_{\R^N} |\nabla v_n|^2 \, dx = \frac{N}{2} \int_{\R^N} H(v_n) \, dx = \frac{N}{2} s_n^{-N} \int_{\R^N} H(s_n^{N/2} \rho_n u) \, dx
$$
	and
$$
\int_{\R^N} |\nabla u|^2 \, dx = \frac{N}{2} s_n^{-N-2} \rho_n^{-2} \int_{\R^N} H(s_n^{N/2} \rho_n u) \, dx = \frac{N}{2} \rho_n^{4/N} \int_{\R^N} \frac{H(s_n^{N/2} \rho_n u)}{ \left|  s_n^{N/2} \rho_n u \right|^{2_N} } |u|^{2_N} \, dx.
$$
There follows that
	$$
	\lim_n\int_{\R^N} \frac{H(s_n^{N/2} \rho_n u)}{ \left|  s_n^{N/2} \rho_n u \right|^{2_N} } |u|^{2_N} \, dx = 0,
	$$
	whence $s_n^{N/2} \rho_n \to 0$. Fix $\varepsilon > 0$. From (A5) and the fact that $\lim_{t\to0}G(t)/|t|^{2^*}=\infty$, there follows that
	$$
	H(s) \geq \frac{4}{N} G(s) \geq \eps^{-1} |s|^{2^*}
	$$
	for sufficiently small $|s|$. Then, taking into account that $u \in L^\infty (\R^N)$, for sufficiently large $n$
\begin{equation*}\begin{split}
		\int_{\R^N} |\nabla u|^2 \, dx & = \frac{N}{2} s_n^{-N-2} \frac{1}{\rho_n^2} \int_{\R^N} H(s_n^{N/2} \rho_n^2 u) \, dx \geq \eps^{-1} \frac{N}{2} s_n^{-N-2} \frac{1}{\rho_n^2} \left| s_n^{N/2} \rho_n \right|^{2^*} |u|_{2^*}^{2^*}\\
		& = \eps^{-1} \frac{N}{2} (s_n \rho_n)^{\frac{4}{N-2}}|u|_{2^*}^{2^*} 
\end{split}\end{equation*}
	and $s_n \rho_n \to 0$ as $n \to \infty$, which completes the proof.
	
	\textit{(b)} Take any $u_0\in\cD\setminus\{0\}$ and note that \eqref{eq:eta} holds. In view of Lemma \ref{lem:phi} there exists $s_0>0$ such that $s_0\star u_0\in\cM$ and
	\begin{equation*}
		c\leq  J(s_0\star u_0)\leq \max_{s>0}J(s\star u_0)\leq \max_{s>0}\frac{s^2}{2}\int_{\rn}|\nabla u_0|^2\,dx-\int_{\rn}\frac{\wt G(s^{N/2}u_0)}{s^N}\,dx.
	\end{equation*}
Observe that the latter expression is finite due to Lemma \ref{lem:phi} with $\theta=0$.
Hence we can take $\theta_0>0$ so small that, if $\theta_i<\theta_0$, then $\sum_{j=1}^K\theta_j^{1-N/2}S^{N/2}/N\geq \theta_i^{1-N/2}S^{N/2}/N$ is greater than the right-hand side of the formula above.
	\end{proof}
We give explicit examples of nonlinearities that do not satisfy the assumptions of Proposition \ref{lem:cless}. Let $N=3$ and $\eps>0$ be sufficiently small. If $\wt g(u)=\wt g_1(u)=\min\{|u|^{4-\eps},|u|^{4/3}\}u$ and if $\theta=\theta_1$ is not sufficiently small, then we can use Lemma \ref{lem:clessR} (a) provided that $\rho=\rho_1$ is sufficiently large, but not part (b). If $G$ is of the form \eqref{eq:Gsp} and
\begin{equation}\label{Ex33}
\wt g_i(u)=\min\{|u|^{4},|u|^{4/3+\eps}\}u
\end{equation}
and if $K=2$ or $\rho$ is not sufficiently large, then we can use Lemma \ref{lem:clessR} (b) provided that $\theta_i$ is sufficiently small for some $i\in\{1\dots,K\}$, but not part (a).

In view of Lemma \ref{lem:coerc}, any minimizing sequence $(u^{(n)})\subset \cM\cap\cD$ such that $J(u^{(n)})\to c>0$ is bounded. By the standard concentration-compactness argument \cite{Lions84}, $u^{(n)}\weakto \tilde u$ for some $\tilde u\neq 0$ up to a subsequence and up to translations. It is not clear, however, if $J(\tilde u)=c$ or $\tilde u\in \cM\cap \D$. Note that we can find $R>0$ such that  $ \tilde u(R\cdot)\in\cM$ and in order to ensure that $J(\tilde u)=c$ and $\tilde u\in\cD$ we need to know that $R\geq 1$. The latter crucial condition requires the profile decomposition analysis of $(u^{(n)})$ provided by the following lemma.

\begin{Lem}\label{lem:split}
	Let $(u^{(n)})\subset H^1(\rn)^K$ be bounded. Then there exist sequences $(\tilde{u}^{(i)})_{i=0}^\infty\subset H^1(\rn)^K$ and $(y^{(i,n)})_{i=0}^\infty\subset\rn$ such that $y^{(0,n)}=0$, $\lim_n|y^{(i,n)}-y^{(j,n)}|=0$ if $i\ne j$, and for every $i\ge0$ and every $F\colon\rn\to\R$ of class $\cC^1$ such that
	\[
	\lim_{u\to0}\frac{F(u)}{|u|^2}=\lim_{|u|\to\infty}\frac{F(u)}{|u|^{2^*}}=0
	\]
	there holds (up to a subsequence)
	\begin{eqnarray}
		u^{(n)}(\cdot+y^{(i,n)})&\rightharpoonup&\tilde{u}^{(i)}\text{ as }n\to\infty \label{eq:Me1}\\
		\lim_n\int_{\rn}|\nabla u^{(n)}|^2\,dx&=&\sum_{j=0}^i\int_{\rn}|\nabla\tilde{u}^{(j)}|^2\,dx+\lim_n\int_{\rn}|\nabla v^{(i,n)}|^2\,dx \label{eq:Me2}\\
		\limsup_n\int_{\rn}F(u^{(n)})\,dx&=&\sum_{i=0}^\infty\int_{\rn}F(\tilde{u}^{(i)})\,dx, \label{eq:Me3}
	\end{eqnarray}
	where $v^{(i,n)}(x):=u^{(n)}(x)-\sum_{j=0}^{i}\tilde{u}^{(j)}(x-y^{(j,n)})$.
\end{Lem}
\begin{proof}
	We argue similarly as in the case $K=1$ provided in \cite[Theorem 1.4]{NonradMed}.
\end{proof}

	\begin{Lem}\label{lem:infismin}
		If (A1)--(A5) and \eqref{eq:eta2} hold and either $\theta=0$ or $\theta\in(0,\infty)^K$ and \eqref{eq:groundlevel} is satisfied, then $c$ is attained.
	\end{Lem}	
	\begin{proof}
		Let $(u^{(n)})\subset\cM\cap\cD$ such that $\lim_nJ(u^{(n)})=c$. Then $(u^{(n)})$ is bounded due to Lemma \ref{lem:coerc} and, in view of Lemma \ref{lem:split}, we find $(\tilde{u}^{(i)})_{i=0}^\infty\subset H^1(\rn)^K$ and $(y_n^{(i,n)})_{i=0}^\infty\subset\rn$ such that \eqref{eq:Me1}--\eqref{eq:Me3} hold. Let $I:=\{i\ge0:\tilde{u}^{(i)}\ne0\}$. 
		
		Suppose that $\theta\in(0,\infty)^K$ and \eqref{eq:groundlevel} is satisfied. 
		
		{\em Claim 1.} $I\neq\emptyset$. By contradiction
		suppose that $\tilde{u}^{(i)}=0$ for every $i\ge0$. Then
		$$\int_{\R^N} |\nabla u^{(n)}|^2 \, dx=\frac{N}{2} \int_{\R^N} H(u^{(n)}) \, dx 
		=\frac{N}{2} \int_{\R^N} \wt H(u^{(n)}) \, dx+\sum_{j=1}^K\theta_j \int_{\R^N} |u^{(n)}_j|^{2^*} \, dx.$$
		Observe that (A1), (A3), and (A5) imply that 
		\[
		\lim_{u\to0}\frac{\wt H(u)}{|u|^2}=\lim_{|u|\to\infty}\frac{\wt H(u)}{|u|^{2^*}}=0
		\]
		and
			\begin{equation}\label{eq:contradict1}
			o(1)+\int_{\R^N} |\nabla u^{(n)}|^2 \, dx=\sum_{j=1}^K\theta_j\int_{\R^N} |u^{(n)}_j|^{2^*} \, dx.
		\end{equation}
		For the sake of simplicity, let us denote $\bar{S}:=\left(\sum_{j=1}^K\theta_j^{1-N/2}\right)^{2/N}S$, cf. Appendix \ref{AppB}. Then
		\begin{equation*}
o(1)+\int_{\R^N} |\nabla u^{(n)}|^2 \, dx\leq \bar{S}^{-2^*/2}\left(\int_{\R^N} |\nabla u^{(n)}|^{2} \, dx\right)^{2^*/2}.
		\end{equation*}
	Passing to a subsequence we set $\nu:=\lim_n\int_{\R^N} |\nabla u^{(n)}|^2 \, dx>0$ from Lemma \ref{lem:bdaw1} and we get $\nu^{2/(N-2)}\ge \bar{S}^{N/(N-2)}$. Then
	\begin{equation}\label{eq:ONE}
	c=\lim_nJ(u^{(n)})=\lim_nJ(u^{(n)})-\frac{1}{2^*}M(u^{(n)})=\frac1N\nu\geq \frac1N \bar{S}^{N/2},
	\end{equation}
so we obtain a contradiction and $I\neq \emptyset$.
		
		{\em Claim 2.} For every $i\in I$ there holds $u^{(n)}(\cdot+y^{(i,n)})\to \tilde{u}^{(i)}$ in $\cD^{1,2}(\R^N)^K$ or
		$\int_{\R^N} |\nabla \tilde{u}^{(i)}|^2 \, dx<\frac{N}{2}\int_{\mathbb{R}^N}H(\tilde{u}^{(i)})\, dx$.
		Suppose that there exists $i\in I$ such that $\nu:=\lim_n\int_{\R^N} |\nabla v^{(n)}|^{2} \, dx>0$ (passing to a subsequence) and the reverse inequality holds, where $v^{(n)}:=u^{(n)}(\cdot+y^{(i,n)})-\tilde{u}^{(i)}$. 	By Vitali's convergence theorem 
		\begin{eqnarray*}
			\int_{\mathbb{R}^N}\big(H(u^{(n)})-H(v^{(n)})\big)\, dx
			&=&\int_{\mathbb{R}^N}\int_0^1 -\frac{d}{ds}H(u^{(n)}-s\tilde{u}^{(i)})\, ds\,dx\\\nonumber
			&=&\int_{\mathbb{R}^N}\int_0^1 h(u^{(n)}-s\tilde{u}^{(i)})\tilde{u}^{(i)}\,ds\, dx\\\nonumber
			&\rightarrow& \int_0^1 \int_{\mathbb{R}^N} h(\tilde{u}^{(i)}-s\tilde{u}^{(i)})\tilde{u}^{(i)}\,dx\,ds\\\nonumber
			&=&\int_{\mathbb{R}^N}\int_0^1 -\frac{d}{ds}H(\tilde{u}^{(i)}-s\tilde{u}^{(i)})\, ds\, dx\\
			&=&\int_{\mathbb{R}^N}H(\tilde{u}^{(i)})\, dx\nonumber
		\end{eqnarray*}
		as $n\to\infty$. Again, passing to a subsequence,
		$$\int_{\R^N} |\nabla v^{(n)}|^{2} \, dx+\int_{\R^N} |\nabla \tilde{u}^{(i)}|^2 \, dx=\frac{N}{2}\Big(\int_{\mathbb{R}^N}H(v^{(n)})\, dx+\int_{\mathbb{R}^N}H(\tilde{u}^{(i)})\, dx\Big)+o(1)$$
		and, since $\int_{\R^N} |\nabla \tilde{u}^{(i)}|^2 \, dx\geq \frac{N}{2}\int_{\mathbb{R}^N}H(\tilde{u}^{(i)})\, dx$, we obtain
		\begin{equation}\label{eq:Rn}
		\int_{\R^N} |\nabla v^{(n)}|^{2} \, dx\leq \frac{N}{2}\int_{\mathbb{R}^N}H(v^{(n)})\, dx+o(1)
		\end{equation}
		and define $R_n>0$ such that $v^{(n)}(R_n\cdot)\in\cM$. We want to prove that $R_n\to1$. If
		$$\frac{N}{2}\int_{\R^N}H(v^{(n)})\,dx<\int_{\R^N}|\nabla v^{(n)}|^2\,dx$$
		holds for a.e. $n$,
		then from \eqref{eq:Rn} and the fact that $\nu>0$ we get $R_n\to1$. If, passing to a subsequence,
		$$\int_{\R^N} |\nabla v^{(n)}|^{2} \, dx\leq \frac{N}{2}\int_{\mathbb{R}^N}H(v^{(n)})\, dx$$
		holds, then we infer $R_n\ge1$. Note that $\lim_n | u^{(n)}|_2^2-| v^{(n)}|_2^2=| \tilde{u}^{(i)}|_2^2>0$, hence $v^{(n)}\in\cD$ and $v^{(n)}(R_n\cdot)\in\cM\cap\cD$ for a.e. $n$. Hence the Brezis--Lieb Lemma yields
		\begin{equation}\label{eq:TWO}\begin{split}
			c&\le J\bigl(v^{(n)}(R_n\cdot)\bigr)=J\bigl(v^{(n)}(R_n\cdot)\bigr)-\frac12M\bigl(v^{(n)}(R_n\cdot)\bigr)\,dx=\frac{1}{R^N_n}\int_{\rn}\frac{N}{4}H(v^{(n)})-G(v^{(n)})\,dx\\
			&\leq \int_{\rn}\frac{N}{4}H(v^{(n)})-G(v^{(n)})\,dx\leq \int_{\rn}\frac{N}{4}H(u^{(n)})-G(u^{(n)})\,dx+o(1)\\
			&=J(u^{(n)})-\frac12M(u^{(n)})+o(1)
			=J(u^{(n)})+o(1)=c+o(1),
		\end{split}\end{equation}
		which implies that $R_n\to 1$ as claimed. Therefore we have that
		\begin{equation}\label{eq:contradict2}
		\int_{\R^N} |\nabla v^{(n)}|^{2} \, dx=o(1)+\frac{N}{2}\int_{\mathbb{R}^N}H(v^{(n)})\, dx=o(1)+\sum_{j=1}^K\theta_j\int_{\R^N} |v_j^{(n)}|^{2^*} \, dx	
		\end{equation}
		and as in {\em Claim 1} we get $\nu^{2/(N-2)}\ge \bar{S}^{N/(N-2)}$. Since $J(u^{(n)})-J(v^{(n)})=J(\tilde{u}^{(i)})+o(1)$ and $J(\tilde{u}^{(i)})\ge\int_{\rn}\frac{N}4H(\tilde{u}^{(i)})-G(\tilde{u}^{(i)})\,dx\ge0$, we have
		\begin{equation}\label{eq:THREE}
		c=\lim_nJ(\tilde{u}^{(i)})+J(v^{(n)})=J(\tilde{u}^{(i)})+\frac12\nu-\frac1{2^*}\lim_n\sum_{j=1}^K\theta_j\int_{\R^N}|v_j^{(n)}|^{2^*}\,dx\geq \frac1N\nu\ge\frac1N\bar{S}^{N/2},
		\end{equation}
		a contradiction.\\
		\indent {\em Conclusion.}	Let $i\in I$ and, for simplicity, let us denote $\tilde{u}^{(i)}=:\tilde{u}$. If $\int_{\R^N} |\nabla \tilde{u}|^2 \, dx<\frac{N}{2}\int_{\mathbb{R}^N}H(\tilde{u})\, dx$, then there 
		exists $R>1$ such that $\tilde u(R\cdot)\in\cM$,  whence $\tilde{u}(R\cdot)\in\cD$.  Hence Fatou's Lemma yields
		\begin{equation}\label{eq:FOUR}\begin{split}
			c&\le J\bigl(\tilde u(R\cdot)\bigr)=J\bigl(\tilde u(R\cdot)\bigr)-\frac12M\bigl(\tilde u(R\cdot)\bigr)\,dx=\frac{1}{R^N}\int_{\rn}\frac{N}{4}H(\tilde u)-G(\tilde{u})\,dx\\
			&<\liminf_n\int_{\rn}\frac{N}{4}H(u^{(n)})-G(u^{(n)})\,dx=\liminf_nJ(u^{(n)})-\frac12M(u^{(n)})=\liminf_nJ(u^{(n)})=c,
		\end{split}\end{equation}
		which is a contradiction. Therefore $u^{(n)}(\cdot+y^{(i,n)})\to \tilde{u}$ in $\cD^{1,2}(\R^N)^K$ (which, together with \eqref{eq:Me2}, implies that $I$ is a singleton) and, consequently, in $L^{2^*}(\rn)^K$. Moreover, in virtue of the Brezis--Lieb lemma, $\int_{\rn}H(u^{(n)})\,dx\to\int_{\rn}H(\tilde{u})\,dx$ because, from the interpolation inequality,
		\[\begin{split}
		\int_{\rn}H(u^{(n)}-\tilde{u})\,dx & \le C(|u^{(n)}-\tilde{u}|_{2_N}^{2_N}+|u^{(n)}-\tilde{u}|_{2^*}^{2^*})\\
		& \le C(|u^{(n)}-\tilde{u}|_2^{2t}|u^{(n)}-\tilde{u}|_{2^*}^{2^*(1-t)}+|u^{(n)}-\tilde{u}|_{2^*}^{2^*})\to0
		\end{split}\]
		for some $C>0$ and $t=\frac{2^*-2_N}{2^*-2}$. Hence $\tilde{u}\in\cM\cap\cD$, and, arguing as before but with $R=1$, $J(\tilde{u})=c$.
		
		Now we consider the case $\theta=0$ and in a similar way we prove {\em Claim 1} and {\em Claim 2} by getting a contradiction in \eqref{eq:contradict1} and \eqref{eq:contradict2}. Finally note that arguments of  {\em Conclusion} apply in the case $\theta=0$ as well.
	\end{proof}

	For $f\colon\rn\to\R$ measurable we denote by $f^*$ the Schwarz rearrangement of $|f|$. Likewise, if $A\subset\rn$ is measurable, we denote by $A^*$ the Schwarz rearrangement of $A$ \cite{BerLions1,LiebLoss}.

	\begin{Lem}\label{lem:radmin}
		Assume that (A1)--(A5) and \eqref{eq:eta2} are verified, $G$ is of the form \eqref{eq:Gsp}, and either $\theta=0$ or $\theta\in(0,\infty)^K$ and \eqref{eq:groundlevel} holds.
		Then $c$ is attained by a $K$-tuple of radial, nonnegative and radially nonincreasing functions.
	\end{Lem}
	\begin{proof}
		Let $\tilde{u}\in\cM\cap\cD$ such that $J(\tilde{u})=c$ be given by Lemma \ref{lem:infismin}.
		For every $j\in\{1,\dots,K\}$ let $u_j$ be the Schwarz rearrangement of $|\tilde{u}_j|$ and denote $u:=(u_1,\dots,u_K)$. 
		Let $a=a(u)$ be determined by Lemma \ref{lem:phi}. In view of the properties of the Schwarz rearrangement  \cite{BerLions1,LiebLoss}, we obtain
		\[
		M(1\star u)=M(u)\le M(\tilde{u})=0,
		\]
		therefore in view of Lemma \ref{lem:phi} we have that $a\le1$ and, consequently, $M(a\star\tilde{u})\ge0$. Let $$d:=\frac{N}{2}\max_{j=1,\dots, L}\Big(\sum_{i=1}^Kr_{i,j}-2\Big)\ge2.$$ Then
		\[\begin{split}
			c\le \, & J(a\star u)=J(a\star u)-\frac{1}{d}M(a\star u)\\
			= \, &\irn\sum_{i=1}^Ka^2\biggl(\frac12-\frac{1}{d}\biggr)|\nabla u_i|^2+\frac{1}{a^N}\biggl(\frac{N}{2d}H_i(a^{N/2}u_i)-G_i(a^{N/2}u_i)\biggr)\,dx\\
			&-\frac{1}{a^N}
			\sum_{j=1}^L\beta_j\bigg(1-\frac1d\Big(\sum_{i=1}^Kr_{i,j}-2\Big)\bigg)\prod_{i=1}^K|a^{N/2}u_i|^{r_{i,j}}\\
			\le \, &\irn\sum_{i=1}^Ka^2\biggl(\frac12-\frac{1}{d}\biggr)|\nabla\tilde{u}_i|^2+\frac{1}{a^N}\biggl(\frac{N}{2d}H_i(a^{N/2}|\tilde{u}_i|)-G_i(a^{N/2}|\tilde{u}_i|)\biggr)\,dx\\
			&-\frac{1}{a^N}
			\sum_{j=1}^L\beta_j\bigg(1-\frac1d\Big(\sum_{i=1}^Kr_{i,j}-2\Big)\bigg)\prod_{i=1}^K|a^{N/2}\tilde{u}_i|^{r_{i,j}}\\
			= \, &J(a\star\tilde{u})-\frac{1}{d}M(a\star\tilde{u})\le J(a\star\tilde{u})\le J(\tilde{u})=c,
		\end{split}\]
		i.e., $J(a\star u)=c$.
	\end{proof}

	\begin{Lem}\label{lem:minons}
		(a) Assume that (A1)--(A3), (A4,$\preceq$), (A5), and \eqref{eq:eta2} hold and let $u\in\cM\cap\cD$ such that $J(u)=c$ and $u_i$ is radial for every $i\in\{1,\dots,K\}$. Then  $u$ is of class $\cC^2$.\\
		(b) If, in addition, $N\in\{3,4\}$, $G$ is of the form \eqref{eq:Gsp}, and $u_i$ is nonnegative for every $i\in\{1,\dots,K\}$, then $u\in\partial\cD$. Moreover, for every $i\in\{1,\dots,K\}$, either $|u_i|_2=\rho_i$ or $u_i=0$.
	\end{Lem}
	\begin{proof}
		\textit{(a)} 
		In Proposition \ref{Prop:Lag} we set $f=J$,  $\phi_i(v)=|v_i|_2^2-\rho_i^2$, $1\leq i\leq m=K$, $\psi_1(v)=M(v)$, $n=1$, $v\in \cH=H^1(\R^N)^K$. Then there exist $(\lambda_1,\dots,\lambda_K)\in[0,\infty)^K$ and $\sigma\in\R$ such that
		\begin{equation}\label{eq:solution}
			-(1-2\sigma)\De u_i+\lambda_iu_i=\partial_iG(u)-\sigma\frac{N}{2}\partial_iH(u)
		\end{equation}
		for every $i\in\{1,\dots,K\}$ and $u$ satisfies the Nehari identity
		\begin{equation}\label{eq:sigN}
			(1-2\sigma)\irn|\nabla u|^2\,dx+\sum_{i=1}^K\irn\la_i|u_i|^2\,dx+\irn\sigma\frac{N}{2}\langle h(u),u\rangle-\langle g(u),u\rangle\,dx=0.
		\end{equation}
		If $\sigma=\frac12$, then (A4,$\preceq$), (A5), and \eqref{eq:sigN} yield
		\[\begin{split}
			0&\ge\irn\frac{N}{4}\langle h(u),u\rangle-\langle g(u),u\rangle\,dx=\irn\frac{N}{4}\langle h(u),u\rangle-H(u)-2G(u)\,dx\\
			&>\irn\frac{N}{2}H(u)-2G(u)\,dx\ge0,
		\end{split}\]
		a contradiction. Hence $\sigma\ne\frac12$ and $u$ satisfies also the Poho\v{z}aev identity
		\begin{equation}\label{eq:sigP}
			(1-2\sigma)\irn|\nabla u|^2\,dx+\frac{2^*}{2}\sum_{i=1}^K\irn\la_i|u_i|^2\,dx+2^*\irn\sigma\frac{N}{2}H(u)-G(u)\,dx=0.
		\end{equation}
		Combining \eqref{eq:sigN} and  \eqref{eq:sigP} we obtain
		\[
		(1-2\sigma)\irn|\nabla u|^2\,dx+\frac{N}{2}\irn\sigma N\Bigl(\frac12\langle h(u),u\rangle-H(u)\Bigr)-H(u)\,dx=0
		\]
		and, using the fact that $u\in\cM$,
		\[
		(1-2\sigma)\irn H(u)\,dx+\irn\sigma N\Bigl(\frac12\langle h(u),u\rangle-H(u)\Bigr)-H(u)\,dx=0,
		\]
		that is
		\[
		\sigma\irn\langle h(u),u\rangle-2_NH(u)\,dx=0,
		\]
		which together with (A4,$\preceq$) yields $\sigma=0$. In view of \cite[Theorem 2.3]{BrezisLiebCMP84}, $u\in W^{2,q}_\textup{loc}(\R^N)^K$ for all $q < \infty$, hence $u\in \cC^{1,\alpha}_\textup{loc}(\rn)^K$ for all $\alpha<1$. Then, arguing as in the proof of \cite[Lemma 1]{BerLions1}, we have that $u$ is of class $\cC^2$.
		
		\textit{(b)} First we show that  $u\in\partial\cD$. Suppose by contradiction that $|u_i|_2<\rho_i$ for every $i$. Then
		 $\la_1=\dots=\la_K=0$ and from \eqref{eq:sigN} and \eqref{eq:sigP} (with $\sigma=0$ as in proof of (a)) there follows
		\begin{equation}\label{eq:Gg2^*}
			\irn\langle g(u),u\rangle-2^*G(u)\,dx=0.
		\end{equation}
		In view of (A5) 
		\begin{equation}\label{eq:pointSob}
			2^* G\bigl(u(x)\bigr)=\langle g\bigl(u(x)\bigr),u(x)\rangle
		\end{equation}
		for every $x\in\R^N$. Since $G_i$ satisfies (A5), we get $2^* G_i(u_i(x))\geq g_i(u_i(x))u_i(x)$ for all $i\in\{1,\dots,K\}$ and note that
		$$2^*\sum_{j=1}^L\beta_j\prod_{i=1}^K|u_i(x)|^{r_{i,j}}\geq \sum_{j=1}^L\beta_j\sum_{k=1}^K r_{k,j}\prod_{i=1}^K|u_i(x)|^{r_{i,j}},$$
		since $\sum_{k=1}^K r_{k,j}<2^*$.
		Hence, from \eqref{eq:pointSob}, the inequalities above are actually equalities. On the other hand, for every $j\in\{1,\dots,L\}$, $\sum_{i=1}^Kr_{i,j}<2^*$, which yields $\beta_j=0$ or $\prod_{i=1}^K|u_i(x)|^{r_{i,j}}=0$ for every $x\in\rn$, so that the coupling term is zero and thus $$2^* G_i(u_i(x))= g_i(u_i(x))u_i(x)$$ for every $i\in\{1,\dots,K\}$ and every $x\in\rn$.
		
		Now fix $i\in\{1,\dots,K\}$ such that $u_i\neq 0$. Since $u_i\in H^1(\R^N)\cap \cC^2$, there exists an open interval $I \subset \R$ such that $0 \in \overline{I}$ and $2^* G_i(s)=g_i(s)s$ for $s\in \overline{I}$.
		Then $G_i(s)=\theta_i|s|^{2^*}/2^*$ for $s\in \overline{I}$ and $u_i$ solves
		$-\Delta u_i = \theta_i |u_i|^{2^* - 2 } u_i$. Hence, since $u_i\ge0$, $u_i$ is an Aubin--Talenti instanton, up to scaling and translations, which is not $L^2$-integrable because $N \in \{3,4\}$. 
		Therefore  $u\in\partial\cD$.
		
		Now we prove the second part and suppose that there exists $\nu\in\{1,\dots,K-1\}$ such that, up to changing the order, $|u_i|_2<\rho_i$ for every $i\in\{1,\dots,\nu\}$ and $|u_i|_2=\rho_i$ for every $i\in\{\nu+1,\dots,K\}$.
		From Proposition \ref{Prop:Lag} there exist $0=\lambda_1=\dots=\lambda_\nu\leq \lambda_{\nu+1},\dots,\lambda_K$ and $\sigma\in\R$ such that
		\begin{equation}\label{eq:lagSM1}
			\begin{cases}
				-(1-2\sigma)\De u_i=\partial_iG(u)-\sigma\frac{N}{2}\partial_iH(u) \quad \text{for every }i\in\{1,\dots,\nu\}\\
				-(1-2\sigma)\De u_i+\lambda_iu_i=\partial_iG(u)-\sigma\frac{N}{2}\partial_iH(u) \quad \text{for every }i\in\{\nu+1,\dots,K\}
			\end{cases}
		\end{equation}
		and as before we obtain	
		$\sigma=0$. Since $G_i$ satisfies the scalar variant of (A5), $(0,\infty)\ni s\mapsto G_i(s)/s^{2_N}\in\R$ is nondecreasing, hence
		$G_i$ is nondecreasing as well for all $i$. Then, the first $\nu$ equations in \eqref{eq:lagSM1} with $\sigma=0$ yield $-\De u_i\ge0$ for $i\in\{1,\dots,\nu\}$. Since  $u\in L^{\frac{N}{N-2}}(\rn)^K$ as $N\in\{3,4\}$, $u$ is of class $\cC^2$, and $u_i\geq 0$,  \cite[Lemma A.2]{Ikoma} implies $u_i=0$ for every $i\in\{1,\dots,\nu\}$. Notice that we have proved that $\la_i=0$ implies that $u_i=0$.
	\end{proof}
	
	\begin{Rem}\label{rem} We point out that in addition to the assumptions of Lemma \ref{lem:minons}, i.e., (A1)--(A3), (A4,$\preceq$), (A5), and \eqref{eq:eta2} hold, $u\in\cM\cap\cD$, and $J(u)=c$, we can show that  $u\in\partial \cD$ for any dimension $N\geq 3$ and without the assumption that $G$ is of the form \eqref{eq:Gsp} provided that  $H\preceq(2^*-2)G$ holds. Indeed, observe that \eqref{eq:Gg2^*} contradicts $H\preceq(2^*-2)G$ and Lemma \ref{Lem:ineq}.
	\end{Rem}

	\begin{proof}[Proof of Theorem \ref{th:main1}]
		Statement (a) follows from Lemmas \ref{lem:infismin} and \ref{lem:radmin}. Now we prove statement (b). From Lemma \ref{lem:minons} (a), $u$ is of class $\cC^2$, while from Proposition \ref{Prop:Lag} there exist $(\la_1,\dots,\la_K)\in[0,\infty)^K$ and $\sigma\in\R$ such that \eqref{eq:solution} holds and $\sigma=0$ as in the proof of Lemma \ref{lem:minons} (a).
	\end{proof}

	\begin{proof}[Proof of Theorem \ref{th:main2}]
		It follows from Lemma \ref{lem:minons} (b), Theorem \ref{th:main1} (b), and the maximum principle \cite[Lemma IX.V.1]{Evans} (the implication $u_i\ne0\Rightarrow\la_i>0$ is proved as in the proof of Lemma \ref{lem:minons} (b)).
	\end{proof}

\begin{proof}[Proof of Corollary \ref{Cor}]
From Theorem \ref{th:main1} \textit{(b)}, there exists $u\in\cM\cap\cD\cap\cC^2(\rn)$ and $\lambda\ge0$ such that $J(u)=c$ and $(\lambda,u)$ is a solution to \eqref{eq:}. Observe that, from Lemma \ref{lem:radmin}, we can assume that $u$ is radial, nonnegative (in fact, positive owing to the maximum principle and because $G$ is nondecreasing on $(0,\infty)$), and radially nonincreasing provided that $G$ is even. Next, since $N\in\{3,4\}$ and $G$ is even or $H\preceq (2^*-2)G$, arguing as in the proof of Lemma \ref{lem:minons} (b) -- see also Remark \ref{rem} -- we obtain that $u\in\partial\cD=\cS$ and $(\la,u)$ is a solution to \eqref{eq:rho}. Since $u$ satisfies the Nehari and the Poho\v{z}aev inequalities, we get
\[
\lambda\frac{2}{N-2}\int_{\rn}|u|^2\,dx=\int_{\rn}2^*G(u)-g(u)u\,dx
\]
and, again arguing as in the proof of Lemma \ref{lem:minons} (b) or Remark \ref{rem}, we obtain $\irn2^*G(u)-g(u)u\,dx>0$, whence $\la>0$. Finally, suppose that $G$ is even, so $u$ is (in particular) positive and radially nonincreasing. Note that $u(x)\to0$ as $|x|\to\infty$ and that there exists $t_0>0$ such that $g(t)\le\lambda t$ for every $t\in[0,t_0]$ and $g(t)>\lambda t$ for every $t>t_0$. If $u$ is constant in the annulus $A:=\{\tau_1<|x|<\tau_2\}$ for some $\tau_2>\tau_1>0$, then $0=-\Delta u=g(u)-\lambda u$ in $A$, thus $-\Delta u\le0$ in $\Omega:=\{|x|>\tau_1\}$ because $u$ is radially nonincreasing and $u(x)\leq t_0$ if $x\in\Om$. At the same time, $u$ attains the maximum over $\overline{\Omega}$ at every point of $A$, which is impossible because $u|_\Omega$ is not constant. This proves that $u$ is radially decreasing.
\end{proof}
	
	\begin{Lem}\label{lem:minons2}
		Suppose that $K=2$, $L=1$, and the assumptions in Lemma \ref{lem:minons} (b) hold. If $r_{1,1}+r_{2,1}>2_N$ and $\beta_1$ is sufficiently large, then $u\in\cS$.
	\end{Lem}
	\begin{proof}
		Since $L=1$, we denote $\beta_1$, $r_{1,1}$, $r_{2,1}$ by  $\beta$, $r_{1}$, $r_{2}$ respectively.
		Suppose by contradiction that $u_1=0$ or $u_2=0$, say $u_1=0$, which implies that $|u_2|_2=\rho_2$. We want to find a suitable $w\in\cS$ such that
		\begin{equation}\label{eq:sup}
			J(a\star w)<c=J(0,u_2),
		\end{equation}
		where $a=a(w)$ is defined in Lemma \ref{lem:phi} (note that $a(w)=b(w)$ because (A4,$\preceq$) holds), which is impossible.
		First we show that $c$ does not depend on $\beta$. Consider the functional
		\[
		J_*\colon v\in H^1(\rn)\mapsto\int_{\rn}\frac12|\nabla v|^2-G_2(v)\,dx\in\R
		\]
		and the sets
		\[\begin{split}
			\cD_* & := \left\{ v \in H^1(\R^N) \ : \ \int_{\R^N} |v|^2 \, dx \leq \rho_2^2 \right\},\\
			\cM_* & := \left\{ v \in H^1(\R^N) \setminus \{0\} \ : \ \int_{\R^N} |v|^2 \, dx = \frac{N}2 \irn H_2(v)\,dx \right\}.
		\end{split}\]
		Observe that $J(0,v)=J_*(v)$ for $v\in H^1(\R^N)$.  Moreover $(0,v)\in\cD$ if and only if $v\in\cD_*$, and $(0,v)\in\cM$ if and only if $v\in\cM_*$. In particular,
		\[
		c=J(0,u_2)=J_*(u_2)\ge\inf_{\cM_*\cap\cD_*}J_*=\inf\{J(0,v):(0,v)\in\cM\cap\cD\}\ge c,
		\]
		i.e., $c=\inf_{\cM_*\cap\cD_*}J_*$, and the claim follows because $J_*$, $\cD_*$, and $\cM_*$ do not depend on $\beta$.
		
		In view of Corollary \ref{Cor}, there exists $\bar{v}\in\cM_*\cap\partial\cD_*$ such that $$J_*(\bar{v})=\inf_{\cM_*\cap\cD_*}J_*=c=\inf_{\cM_*\cap\partial\cD_*}J_*.$$ Note that $\bar{v}$ does not depend on $\beta$. Define $w=(w_1,w_2):=\bigl(\frac{\rho_1}{\rho_2}\bar{v},\bar{v}\bigr)$.
		From Lemma \ref{lem:phi}, $a=a_\beta$ is implicitly defined by
		\[\begin{split}
			\irn|\nabla w|^2\,dx=&\,\frac{N}{2}\irn\frac{G_1'(a_\beta^{N/2}w_1)a_\beta^{N/2}w_1-2G_1(a_\beta^{N/2}w_1)}{a_\beta^{N+2}}+\frac{G_2'(a_\beta^{N/2}w_2)a_\beta^{N/2}w_2-2G_2(a_\beta^{N/2}w_2)}{a_\beta^{N+2}}\\
			&+\beta(r_1+r_2-2)a_\beta^{N(r_1+r_2-2)/2-2}w_1^{r_1}w_2^{r_2}\,dx\\
			\ge&\,\beta(r_1+r_2-2)a_\beta^{N(r_1+r_2-2)/2-2}\frac{N}{2}\irn w_1^{r_1}w_2^{r_2}\,dx,
		\end{split}\]
		hence there exist $C>0$ not depending on $\beta$ such that
		\begin{equation}\label{eq:beta1}
			0<\beta a_\beta^{N(r_1+r_2-2)/2-2}\le C,
		\end{equation}
		whence
		\begin{equation}\label{eq:beta2}
			\lim_{\beta\to\infty}a_\beta=0.
		\end{equation}
		Since $a_\beta\star w\in\cM$, we have from (A5)
		\[\begin{split}
			J(a_\beta\star w)&=\irn\frac{N}{4}H(a_\beta\star w)-G(a_\beta\star w)\,dx\le\frac{2}{N-2}\irn G(a_\beta\star w)\,dx\\
			&=\frac{2}{N-2}\irn\frac{G_1(a_\beta^{N/2}w_1)+G_2(a_\beta^{N/2}w_2)}{a_\beta^N}\,dx+\frac{2\beta a_\beta^{N(r_1+r_2-2)/2}}{N-2}\irn w_1^{r_1}w_2^{r_2}\,dx,
		\end{split}\]
		therefore \eqref{eq:sup} holds true for sufficiently large $\beta$ owing to (A1), \eqref{eq:beta1}, and \eqref{eq:beta2}.
	\end{proof}
	
	
	\begin{proof}[Proof of Theorem \ref{th:main3}]
		It follows from Lemma \ref{lem:minons2} and Theorem \ref{th:main2}.
	\end{proof}
	
	Now we investigate the behaviour of the ground state energy with respect to $\rho$. For $\rho=(\rho_1,\dots,\rho_K)\in(0,\infty)^K$ we denote
	\begin{eqnarray*}
		\cD(\rho)&:=&\biggl\{u\in H^1(\rn)^K:\irn |u_i|^2\,dx\le\rho_i^2 \text{ for every } i\in\{1,\dots,K\}\biggr\}\\
		\cS(\rho)&:=&\biggl\{u\in H^1(\rn)^K:\irn |u_i|^2\,dx=\rho_i^2 \text{ for every } i\in\{1,\dots,K\}\biggr\}\\
		c(\rho)&:=&\inf\{J(u):u\in\cM\cap\cD(\rho)\}.
	\end{eqnarray*}

	\begin{Prop}\label{Prop:GSE}
	Assume that (A0)--(A5) and \eqref{eq:eta2} are satisfied.\\
	(i) If $\theta=0$, then $c$ is continuous and $\lim_{\rho\to0^+}c(\rho)=\infty$, where $\rho\to0^+$ means $\rho_i\to0^+$ for every $i\in\{1,\dots,K\}$.\\
	\noindent (ii) Let $\theta\in(0,\infty)^K$ and $\rho\in(0,\infty)^K$ . If \eqref{eq:groundlevel} holds for every $\rho'\in\prod_{j=1}^K(\rho_j-\eps,\rho_j)$ and some $\eps>0$, then $c$ is continuous at $\rho$. 
	If \eqref{eq:groundlevel} holds for every $\rho'\in(0,\eps)^K$ and some $\eps>0$, then $\displaystyle\lim_{\rho'\to0^+}c(\rho')=
	\frac{1}{N}S^{N/2}\sum_{i=1}^K\theta_i^{1-N/2}$\\
\noindent	
	(iii) If every ground state solution to \eqref{eq:} belongs to $\cS(\rho)$ (e.g. if the assumptions of Theorem \ref{th:main3} are satisfied), then $c$ is decreasing in the following sense: if $\rho,\rho'\in(0,\infty)^K$ are such that $\rho_i\ge\rho_i'$ for every $i\in\{1,\dots,K\}$ and $\rho_j>\rho_j'$ for some $j\in\{1,\dots,K\}$, then $c(\rho)<c(\rho')$.
	\end{Prop}
	\begin{proof}
		Fix $\rho\in(0,\infty)^K$ and let $\rho^{(n)}\to\rho$. We begin by proving the upper semicontinuity of $c$ at $\rho$. Let $w\in\cM\cap\cD(\rho)$ such that $J(w)=c(\rho)$, denote $w_i^{(n)}:=\rho_i^{(n)}w_i/\rho_i$, and consider $w^{(n)}=(w_1^{(n)},\dots,w_K^{(n)})\in\cD(\rho^{(n)})$. Due to Lemma \ref{lem:phi}, for every $n$ there exists $s_n>0$ such that $s_n\star w^{(n)}\in\cM$. Note that
		\begin{equation}\label{eq:conts}
			\frac{N}{2}\irn\frac{H\bigl(s_n^{N/2}(\rho_1^{(n)}w_1/\rho_1,\dots,\rho_K^{(n)}w_K/\rho_K)\bigr)}{s_n^{N+2}}\,dx=\irn|\nabla w^{(n)}|^2\,dx\to \irn|\nabla w|^2\,dx.
		\end{equation}
		If $\limsup_ns_n=\infty$, then from (A2) and (A5) the left-hand side of \eqref{eq:conts} tends to $\infty$ up to a subsequence, which is a contradiction. If $\liminf_ns_n=0$, then from (A1), (A3), (A5) and \eqref{eq:eta2} and arguing as in Lemma \ref{lem:bdaw1} we obtain that the limit superior of the left-hand side of \eqref{eq:conts} is less than $|\nabla w|_2^2$, which is again a contradiction.
		There follows that, up to a subsequence, $s_n\to s$ for some $s>0$ and $s\star w\in\cM$. In view of Lemma \ref{lem:phi},
		\[
		\limsup_nc(\rho^{(n)})\le\lim_nJ(s_n\star w_n)=J(s\star w)=J(w)=c(\rho).
		\]
		
		Now we prove the lower semicontinuity of $c$ at $\rho$. Let $\rho^{(n)}\to\rho$ and $u^{(n)}\in\cM\cap\cD(\rho^{(n)})\subset\cM\cap\cD(2\rho)$ such that $J(u^{(n)})=c(\rho^{(n)})\le c(\rho/2)$. In view of Lemma \ref{lem:coerc}, $(u^{(n)})$ is bounded, hence we can consider the sequences $(\tilde{u}^{(i)})$ and $(y^{(i,n)})$ given by Lemma \ref{lem:split}; note that $\tilde{u}^{(i)}\in\cD$. We consider the case $\theta\in(0,\infty)^K$ because the other one (i.e., $\theta=0$) is similar and simpler.
		
		\textit{Claim:} There exists $i\ge0$ such that $\lim_nu^{(n)}(\cdot+y^{(i,n)})\to\tilde{u}^{(i)}\ne0$ in $\cD^{1,2}(\rn)^K$. The proof is similar to that of Lemma \ref{lem:infismin}, thus we focus only on the differences. If $\tilde{u}^{(i)}=0$ for every $i\ge0$, then as in \eqref{eq:ONE} we obtain the contradiction
		\begin{equation}\label{eq:eq}
		\frac{\bar{S}^{N/2}}{N}>c(\rho_1-\eps,\dots,\rho_K-\eps)\ge\limsup_nc(\rho^{(n)})=\limsup_nJ(u^{(n)})\ge\frac{\bar{S}^{N/2}}{N}.
		\end{equation}
	Let $i\ge0$ such that $\tilde{u}^{(i)}\ne0$ and define $v^{(n)}:=u^{(n)}(\cdot+y^{(i,n)})-\tilde{u}^{(i)}$. If $\liminf_n|\nabla v^{(n)}|_2>0$ and $|\nabla\tilde{u}^{(i)}|_2^2\ge\frac{N}{2}\int_{\rn}H(\tilde{u}^{(n)})\,dx$, then we prove that $R_n\to1$, where $R_n>0$ is such that $v^{(n)}(R_n\cdot)\in\cM$. In particular, if up to a subsequence $R_n\ge1$, then as in \eqref{eq:TWO} we get
	\begin{equation*}\begin{split}
	0<c(2\rho)\le c(\rho^{(n)})&\le\frac{1}{R_n}\int_{\rn}\frac{N}{4}H(v^{(n)})-g(v^{(n)})\,dx\le\int_{\rn}\frac{N}{4}H(v^{(n)})-G(v^{(n)})\,dx\\
	&\le c(\rho^{(n)})+o(1).
	\end{split}\end{equation*}
Next, as in \eqref{eq:THREE} we obtain again the contradiction \eqref{eq:eq}, which proves that $v^{(n)}\to0$ in $\cD^{1,2}(\rn)^K$ (up to a subsequence) or $|\nabla\tilde{u}^{(i)}|_2^2<\frac{N}{2}\int_{\rn}H(\tilde{u}^{(n)})\,dx$. In the latter case, we define $R>1$ such that $\tilde{u}^{(i)}(R\cdot)\in\cM$ as in \eqref{eq:FOUR} we get the contradiction
\begin{equation*}
c(\rho)\le J\bigl(\tilde{u}^{(i)}(R\cdot)\bigr)<\limsup_nc(\rho^{(n)})\le c(\rho),
\end{equation*}
where the last inequality is due to the upper semicontinuity. This proves the \textit{Claim}, which yields, together with the interpolation inequality, that $\tilde{u}^{(i)}\in\cM\cap\cD$ and so
\begin{equation*}
c(\rho)\le J(\tilde{u}^{(i)})=\lim_nJ(u^{(n)})=\lim_nc(\rho^{(n)}).
\end{equation*}
		
		
		Now we prove the behaviour of $c(\rho')$ as $\rho'\to0$. Let $\rho^{(n)}\to0^+$ and $u^{(n)}\in\cM\cap\cD(\rho^{(n)})$ such that $J(u^{(n)})=c(\rho^{(n)})$. Denote $s_n:=|\nabla u^{(n)}|_2^{-1}$ and $w^{(n)}:=s_n\star u^{(n)}$ and note that $s_n^{-1}\star w^{(n)}=u^{(n)}\in\cM$, $|\nabla w^{(n)}|_2=1$ and $$|w^{(n)}|_2^2=|u^{(n)}|_2^2=|\rho^{(n)}|^2\to0$$ as $n\to\infty$. In particular $\bigl(w^{(n)}\bigr)$ is bounded in $L^{2^*}(\rn)^K$ and so
		\[
		|w^{(n)}|_{2_N}\le|w^{(n)}|_2^\frac{2}{N+2}|w^{(n)}|_{2^*}^\frac{N}{N+2}\to0
		\]
		as $n\to\infty$. 
		Suppose that $\theta=0$. Then, in view of (A1) and (A3), for every $s>0$
		\[
		\lim_n\irn\frac{G(s^{N/2}w^{(n)})}{s^{N}}\,dx=0
		\]
		and, consequently,
		\[
		J(u^{(n)})=J(s_n^{-1}\star w^{(n)})\ge J(s\star w^{(n)})=\frac{s^2}{2}-\irn\frac{G(s^{N/2}w^{(n)})}{s^N}\,dx=\frac{s^2}{2}+o(1),
		\]
		whence $\lim_nJ(u^{(n)})=\infty$.
		
		Now suppose that $\theta\in (0,\infty)^K$. Since $|u^{(n)}|_2^2=|\rho^{(n)}|^2\to0$,  we get $u^{(n)}\to 0$ in $L^q(\R^N)^K$ for $2\leq q<2^*$. Arguing as above, for every $s>0$
		\[
		\lim_n\irn\frac{\wt G(s^{N/2}u^{(n)})}{s^{N}}\,dx=0,
		\]
		hence
		\[
		\lim_n\irn\frac{G(s^{N/2}u^{(n)})}{s^{-N}}\,dx=\lim_n\irn\frac{\wt G(s^{N/2}u^{(n)})}{s^{-N}}\,dx+\frac{s^{2^*}}{2^*}\sum_{j=1}^K\theta_j
		\lim_n\irn|u^{(n)}_j|^{2^*}\, dx.
		\]
		Consequently,
		\begin{eqnarray*}
		J(u^{(n)})&\geq &J(s\star u^{(n)})=\frac{s^2}{2}\int_{\R^N}|\nabla u^{(n)}|^2\,dx-\irn\frac{G(s^{N/2}u^{(n)})}{s^N}\,dx\\
		&=&\frac{s^2}{2}\lim_n\int_{\R^N}|\nabla u^{(n)}|^2\,dx-\frac{s^{2^*}}{2^*}\sum_{j=1}^K\theta_j
		\lim_n\irn|u^{(n)}_j|^{2^*}\, dx+o(1)
		\end{eqnarray*}
		for any $s>0$.
		Then, in view Lemma \ref{lem:BarS}
			\begin{eqnarray*}
			\lim_nJ(u^{(n)})&\geq &
			\max_{s>0}\; \frac{s^2}{2}\lim_n\int_{\R^N}|\nabla u^{(n)}|^2\,dx-\frac{s^{2^*}}{2^*}\sum_{j=1}^K\theta_j
			\lim_n\irn|u^{(n)}_j|^{2^*}\, dx\\
			&=&
			\frac{1}{N}\frac{\lim_n|\nabla u^{(n)}|_2^N}{\left(\sum_{j=1}^K\theta_j\lim_n|u_j^{(n)}|_{2^*}^{2^*}\right)^{N/2-1}}\\
			&\geq& \frac{1}{N}\bar{S}^{\frac{N}{2}}=\frac{1}{N}S^{N/2}\sum_{i=1}^K\theta_i^{1-N/2}
		\end{eqnarray*}
	and taking into account \eqref{eq:groundlevel} we obtain
 $$\lim_nJ(u^{(n)})=\frac{1}{N}S^{N/2}\sum_{i=1}^K\theta_i^{1-N/2}.$$

		Now assume that every ground state solution to \eqref{eq:} belongs to $\cS(\rho)$ and let $\rho,\rho'$ as in the statement. Let $u\in\cM\cap\cS(\rho)$ and $u'\in\cM\cap\cS(\rho')\subset\cM\cap\cD(\rho)\setminus\cS(\rho)$ such that $J(u)=c(\rho)$ and $J(u')=c(\rho')$. Clearly $c(\rho)\le c(\rho')$. If $c(\rho)=c(\rho')$, then $c(\rho)=J(u')$, with $u'\in\cM\cap\cD(\rho)\setminus\cS(\rho)$, which is a contradiction.
	\end{proof}
	
	\appendix
	\section{Sign of Lagrange multipliers}
	
	The following result concerns the sign of a Lagrange multiplier when the corresponding constraint is given by an inequality and the critical point of the restricted functional is a minimizer. The result is related with Clarke's \cite[Theorem 1]{Clarke}, however it is not clear whether we can apply it directly in our situation.
	
	\begin{Prop}\label{Prop:Lag}
		Let $\cH$ be a real Hilbert space and $f,\phi_i,\psi_j\in\cC^1(\cH)$, $i\in\{1,\dots,m\}$, $j\in\{1,\dots,n\}$. Suppose that for every $$x\in\bigcap_{i=1}^m\phi_i^{-1}(0)\cap\bigcap_{j=1}^n\psi_j^{-1}(0)$$ the differential $$\bigl(\phi_i'(x),\psi_j'(x)\bigr)_{1\le i\le m,1\le j\le n}\colon\cH\to\R^{m+n}$$ is surjective. If $\bar{x}\in\cH$ minimizes $f$ over
		\[
		\{x\in\cH:\phi_i(x)\le0\text{ for every }i=1,\dots,m\text{ and }\psi_j(x)=0\text{ for every }j=1,\dots,n\},
		\]
		then there exist $(\lambda_i)_{i=1}^m\in[0,\infty)^m$ and $(\sigma_j)_{j=1}^n\in\R^n$ such that
		\[
		f'(\bar{x})+\sum_{i=1}^m \lambda_i\phi_i'(\bar{x})+\sum_{j=1}^n \sigma_i\psi_j'(\bar{x})=0.
		\]
	\end{Prop}
	\begin{proof}
		Fix $\eps>0$ and define the functional $F\colon\cH\to[0,\infty)$ as
		\[
		F(x):=\max_{1\le i\le m,1\le j\le n}\{f(x)-f(\bar{x})+\eps,\phi_i(x),|\psi_j(x)|\}.
		\]
		and observe that $F$ is locally Lipschitz and bounded from below by $0$.
		Since $F(\bar{x})=\eps$, in view of the Ekeland variational principle \cite[Theorem 1.1]{Ekeland} there exists $z=z_\eps\in\cH$ such that
		\begin{eqnarray*}
			\|\bar{x}-z\|&\le&\sqrt{\eps},\\
			F(x)+\sqrt{\eps}\,\|x-z\|&\ge&F(z) \quad \forall x\in\cH.
		\end{eqnarray*}
		From \cite[Propositions 6, 8]{Clarke} there follows that $0\in\partial F(z)+\sqrt{\eps}\,\partial\|\cdot-z\|(z)$, where $\partial$ stands for the generalized gradient \cite[Definition 1]{Clarke}. Hence, there exists $\xi=\xi_\eps\in \partial F(z)$ such that $-\xi\in \sqrt{\eps}\,\partial\|\cdot-z\|(z)$. In view of \cite[Propositions 1, 9]{Clarke},
		$\|\xi\|\leq \sqrt{\eps}$ and $\xi$ lies in the convex hull of $f(z)-f(\bar{x})+\eps$, $\phi_i(z)$, and $|\psi_j(z)$|, i.e., there exists $\tau,\lambda_1,\dots,\lambda_m,\hat{\sigma}_1,\dots,\hat{\sigma}_n\geq 0$ depending on $\eps$, 
		such that $\tau+\lambda_1+\dots+\lambda_m+\hat{\sigma}_1+\dots+\hat{\sigma}_n=1$,
		\[
		\xi\in\biggl(\tau f'(z)+\sum_{i=1}^m\lambda_i\phi_i'(z)+\sum_{j=1}^n\hat{\sigma}_j\partial|\psi_j|(z)\biggr),
		\]
		and $\lambda_i=0$ (resp. $\hat{\sigma}_j=0$) if $\phi_i(z)\le0$ (resp. $\psi_j(z)=0$).
		
		For every $j\in\{1,\dots,n\}$ such that $\psi_j(z)\ne0$ we have
		\[
		\partial|\psi_j|(z)=\{\textup{sign}\bigl(\psi_j(z)\bigr)\psi_j'(z)\}.
		\]
		If $j\in\{1,\dots,n\}$ is as before, we define $\sigma_j:=\textup{sign}\bigl(\psi_j(z)\bigr)\hat{\sigma}_j$, otherwise we define $\sigma_j:=0$. In particular, we have
		\[
		\sum_{j=1}^n\hat{\sigma}_j\partial|\psi_j|(z)=\Biggl\{\sum_{j=1}^n\sigma_j\psi_j'(z)\Biggr\}.
		\]
		
		Summing up, we obtain the following: for every $\eps>0$ there exist $\tau\ge0$, $(\lambda_i)_{i=1}^m\in[0,\infty)^m$, $(\sigma_j)_{j=1}^n\in\R^n$ and $z\in B(\bar{x},\sqrt{\eps})$ such that
		\begin{eqnarray*}
			&\xi:=\tau f'(z)+\sum_{i=1}^m\lambda_i\phi_i'(z)+\sum_{j=1}^n\sigma_j\psi_j'(z)\in B(0,\sqrt{\eps}),\\
			&\tau+\sum_{i=1}^m\lambda_i+\sum_{j=1}^n|\sigma_j|=1.
		\end{eqnarray*}
		Letting $\eps\to0^+$ we get
		\begin{equation}\label{eq:tau}
			\tau f'(\bar{x})+\sum_{i=1}^m\lambda_i\phi_i'(\bar{x})+\sum_{j=1}^n\sigma_j\psi_j'(\bar{x})=0
		\end{equation}
		for some $\tau\ge0$, $(\lambda_i)_{i=1}^m\in[0,\infty)^m$, $(\sigma_j)_{j=1}^n\in\R^n$ such that
		\[
		\tau+\sum_{i=1}^m\lambda_i+\sum_{j=1}^n|\sigma_j|=1.
		\]
		Suppose by contradiction that $\tau=0$, whence
		\begin{equation}\label{eq:app}
			\sum_{i=1}^m\lambda_i\phi_i'(\bar{x})+\sum_{j=1}^n\sigma_j\psi_j'(\bar{x})=0.
		\end{equation}
		If $\phi_i(\bar{x})<0$ for some $i\in\{1,\dots,m\}$, then of course $\lambda_i=0$, hence, up to considering a (possibly empty) subset of $\{1,\dots,m\}$ in \eqref{eq:app}, we can assume that $\phi_1(\bar{x})=\dots=\phi_{m_0}(\bar{x})=0$ and $\lambda_{m_0+1}=\ldots=\lambda_m=0$ for some $0\leq m_0\leq m$, where $m_0=0$ denotes that $\lambda_i=0$ for all $i\in \{1,\dots,m\}$, whereas  $m_0=m$ denotes $\phi_1(\bar{x})=\dots=\phi_{m}(\bar{x})=0$. Then the differential $$\bigl(\phi_1'(\bar{x}),\dots,\phi_{m_0}'(\bar{x}),\psi_1'(\bar{x}),\dots,\psi_n'(\bar{x})\bigr)\colon\cH\to\R^{m_0+n}$$ is surjective and so, for every $i\in\{1,\dots,m_0\}$ (resp. $j\in\{1,\dots,n\}$), we can choose $y\in\cH$ such that $\phi_i'(\bar{x})(y)\ne0$, $\phi_k'(\bar{x})(y)=0$ for every $k\in\{1,\dots,m_0\}\setminus\{i\}$ and $\psi_j'(\bar{x})(y)=0$ for every $j\in\{1,\dots,n\}$ (resp. $\psi_j'(\bar{x})(y)\ne0$, $\psi_k'(\bar{x})(y)=0$ for every $k\in\{1,\dots,n\}\setminus\{j\}$ and $\phi_i'(\bar{x})(y)=0$ for every $i\in\{1,\dots,m_0\}$). This and \eqref{eq:app} implies $\lambda_i=0$ for every $i\in\{1,\dots,m_0\}$ and $\sigma_j=0$ for every $j\in\{1,\dots,n\}$, a contradiction. We can thus divide both sides of \eqref{eq:tau} by $\tau$ and, up to relabelling $\lambda_i$ and $\sigma_j$ ($i\in\{1,\dots,m_0\}$, $j\in\{1,\dots,n\}$), conclude the proof.
	\end{proof}

\section{A Sobolev-type constant}\label{AppB}
	
Let $\theta=(\theta_1,\dots,\theta_K)\in(0,\infty)^K$,
$$\bar S:= \inf_{u\in\cD^{1,2}(\R^N)^K\setminus\{0\}} \frac{\int_{\R^N} |\nabla u|^2 \, dx}{\left(\sum_{j=1}^K\theta_j\int_{\R^N} |u_j|^{2^*} \, dx\right)^{2/2^*}},$$
and, clearly, in view of the Sobolev embeddings, $\bar S>0$.
\begin{Lem}\label{lem:BarS}
$\bar S$ is attained by $(\theta_1^{-\frac{N-2}{4}} u_1,...,\theta_K^{-\frac{N-2}{4}} u_K)$, where  $u_j$ are Aubin--Talenti instantons. Moreover
$$\bar S:=\left(\sum_{j=1}^K\theta_j^{-\frac{N-2}{2}}\right)^{2/N}S.$$
\end{Lem}
\begin{proof}
We prove that $\bar{S}$ is attained. Let $I\colon\cD^{1,2}(\rn)^K\to\R$ be defined as
\[
I(u)=\irn\frac12|\nabla u|^2-\frac1{2^*}\sum_{j=1}^K\theta_j|u_j|^{2^*}\,dx.
\]
If $u=(u_1,...,u_K)\in\cD^{1,2}(\rn)^K$, then
\[
I'(u)=0 \Leftrightarrow -\Delta u_j=\theta_j|u_j|^{2^*-2}u_j \text{ for every } j\in\{1,\dots,K\}.
\]
Define the Nehari manifold for $I$ as
\[
\cN:=\left\{u\in\cD^{1,2}(\rn)^K\setminus\{0\} \, : \, |\nabla u|_2^2=\sum_{j=1}^K\theta_j|u_j|_{2^*}^{2^*}\right\}
\]
ad note that, if $u\in\cN$, then
\[\begin{split}
I(u)&=\frac1N|\nabla u|_2^2=\frac1N\sum_{j=1}^K\theta_j|u_j|_{2^*}^{2^*}\\
\frac{\int_{\R^N} |\nabla u|^2 \, dx}{\left(\sum_{j=1}^K\theta_j\int_{\R^N} |u_j|^{2^*} \, dx\right)^{2/2^*}}&=|\nabla u|_2^{2(2^*-2)/2^*}=\left(\sum_{j=1}^K\theta_j|u_j|_{2^*}^{2^*}\right)^{(2^*-2)/2^*},
\end{split}\]
hence $\frac{\int_{\R^N} |\nabla u|^2 \, dx}{\left(\sum_{j=1}^K\theta_j\int_{\R^N} |u_j|^{2^*} \, dx\right)^{2/2^*}}=A$ if and only if $I(u)=\frac1N A^{2^*/(2^*-2)}$. Moreover, if $u\in\cD^{1,2}(\rn)^K\setminus\{0\}$, then $tu\in\cN$ for some $t>0$ and the fraction in the definition of $\bar{S}$ does not depend on the rescaling $t\mapsto tu$, therefore
\[
\bar S= \inf_{u\in\cN} \frac{\int_{\R^N} |\nabla u|^2 \, dx}{\left(\sum_{j=1}^K\theta_j\int_{\R^N} |u_j|^{2^*} \, dx\right)^{2/2^*}}
\]
and $\inf_\cN I=\frac1N\bar{S}^{2/(2^*-2)^*}>0$. 
Let $u^{(n)}\in\cN$ be such that $I(u^{(n)})\to\inf_\cN I$. Up to replacing $u^{(n)}=(u_1^{(n)},...,u_K^{(n)})$ with $(|u_1^{(n)}|,...,|u_K^{(n)}|)$, we can also assume that $u_j^{(n)}\ge0$ for every $n,j$.  In virtue of Ekeland's variational principle \cite{Willem}, we can assume $I'(u^{(n)})\to0$. Since $\inf_\cN I>0$, $u^{(n)}\not\to0$ in $L^{2^*}(\rn)^K$, thus, in view of Solimini's theorem \cite[Theorem 1]{Solimini}, see also \cite[Lemma 5.3]{Tintarev}), there exist $(s_n)\subset (0,\infty)$, $(y_n)\subset \R^N$ and $u\in\cD^{1,2}(\rn)^K\setminus\{0\}$,  such that $s_n^{1/2}u^{(n)}(s_n\cdot +y_n))\weakto u$ in $\cD^{1,2}(\rn)^K$ and $s_n^{1/2}u^{(n)}(s_n\cdot +y_n))\to u$ a.e. in $\rn$ up to a subsequence. In particular, $I'(u)=0$ and so $u\in\cN$. Observe that each component of $u$ is of the form $u_j=\theta_j^{-\frac{N-2}{4}} u_j^0$, where  $u_j^0$ is an Aubin--Talenti instanton.
Therefore
$$\bar S= \frac{\int_{\R^N} |\nabla u|^2 \, dx}{\left(\sum_{j=1}^K\theta_j\int_{\R^N} |u_j|^{2^*} \, dx\right)^{2/2^*}}
=\left(\sum_{j=1}^K\theta_j^{-\frac{N-2}{2}}\right)^{2/N}S.\qedhere$$
\end{proof}
	
	\section*{Acknowledgements}
	The authors are grateful to the reviewer for the fruitful remarks concerning the first version of this paper.
	The authors were partially supported by the National Science Centre, Poland (Grant No. 2017/26/E/ST1/00817). J. Mederski was also partially supported by the Alexander von Humboldt Foundation (Germany) and by the Deutsche Forschungsgemeinschaft (DFG, German Research Foundation) -- Project-ID 258734477 -- SFB 1173 during the stay at Karlsruhe Institute of Technology.

\end{document}